\documentclass[12pt]{article}

\usepackage[a4paper,margin=3cm,innermargin=3cm]{geometry}
\usepackage{hyperref}
\usepackage[T1]{fontenc}
\usepackage[english]{babel}
\usepackage{csquotes}
\usepackage{indentfirst}

\usepackage{verbatim}

\usepackage{amsmath, amssymb, amsfonts, amsthm}

\theoremstyle{theorem}
\newtheorem{theorem}{Theorem}[section]
\newtheorem{lemma}[theorem]{Lemma}

\newtheorem{corollary}[theorem]{Corollary}
\newtheorem{conjecture}[theorem]{Conjecture}
\newtheorem{claim}[theorem]{Claim}
\newtheorem{prop}[theorem]{Proposition}

\theoremstyle{remark}
\newtheorem{remark}{Remark}[section]

\newcommand{\op}{\operatorname}
\newcommand{\abs}[1]{\left\lvert #1 \right\rvert}

\newcommand{\br}[1]{\left\{ #1 \right\}}
\newcommand{\setdef}[2]{\br{#1 \colon #2}} 
\newcommand{\pr}[1]{\left( #1 \right)} 
\newcommand{\floor}[1]{\left\lfloor #1 \right\rfloor}
\newcommand{\ceil}[1]{\left\lceil #1 \right\rceil}
\newcommand{\irange}[2]{\left[#1,#2\right]} 
\newcommand{\Pb}[1]{\mathbb{P}{\pr{#1}}}

\newcommand{\D}[1]{\op{Var}\!\pr{#1}}
\DeclareMathOperator*{\argmax}{argmax}

\begin{document}
\date{}

\title{Maximum induced trees and forests of bounded degree\\ in random graphs}

\author{Margarita Akhmejanova, \,
Vladislav Kozhevnikov\footnote{Moscow Institute of Physics and Technology (National Research University), Department of Discrete Mathematics, Moscow, Russian Federation}, \, Maksim Zhukovskii\footnote{The University of Sheffield, Department of Computer Science, Sheffield S1 4DP, UK.\newline Email: m.zhukovskii@sheffield.ac.uk.}}
\maketitle



\textbf{Abstract}. Asymptotic behaviour of maximum sizes of induced trees and forests has been studied extensively in last decades, though the overall picture is far from being complete. In this paper, we close several significant gaps: 1) We prove $2$-point concentration of the maximum sizes of an induced forest and an induced tree with maximum degree at most $\Delta$ in dense binomial random graphs $G(n,p)$ with constant probability $p$.  2) We show concentration in an explicit interval of size $o(1/p)$ for the maximum size of an induced forest with maximum degree at most $\Delta$ for $1/n\ll p=o(1)$. Our proofs rely on both the second moment approach, with the probabilistic part involving Talagrand's concentration inequality and the analytical part involving saddle-point analysis, and new results on enumeration of labelled trees and forests that might be of their own interest. 
\vspace{0.3cm}

\vspace{0.5cm}

\section{Introduction}
\label{sc:intro}


All graphs considered in the paper are simple, i.e. undirected, without loops and multiple edges. Let $G$ be a graph. Then $V(G)$ denotes the set of vertices of $G$ and $E(G)$ denotes the set of edges of $G$. A subgraph $H$ of $G$ is called \emph{induced} if, for any two vertices $u,v$ in $H$, $u$ and $v$ are adjacent in $H$ if and only if they are adjacent in $G$. For a subset of vertices $S\subset V(G)$, $G[S]$ denotes the induced subgraph $H$ of $G$ such that $V(H)=S$. We call $|V(G)|$ the \textit{size} of $G$. For a graph $G$, we denote the maximum size of its 
\begin{itemize}
\item induced tree by $\mathrm{T}(G)$, 
\item induced forest by $\mathrm{F}(G)$,
\item induced tree with maximum degree at most $\Delta$ by $\mathrm{T}_{\Delta}(G)$, and
\item induced forest with maximum degree at most $\Delta$  by $\mathrm{F}_{\Delta}(G)$.
\end{itemize}


Let $G(n,p)$ denote the binomial random graph on $[n]:=\{1,\ldots,n\}$, in which each pair of vertices is connected by an edge with probability $p$ independently from all the other pairs of vertices.
For a graph property, we say that it holds \textit{with high probability} (\textit{whp}, for brevity) if  $G(n,p)$ has this property with probability approaching $1$ as $n\to\infty$. 

In this paper, we study asymptotic behaviour of $\mathbf{T}_{\Delta}:=\mathrm{T}_{\Delta}(G(n,p))$ and $\mathbf{F}_{\Delta}:=\mathrm{F}_{\Delta}(G(n,p))$ in different regimes: $p=\mathrm{const}\in(0,1)$ and $1/n\ll p=o(1)$.

Throughout the paper, we use standard asymptotic notations $o,O,\Theta,\Omega$ for sequences and functions (see, e.g., \cite[Chapter A.2]{Flajolet2009} for definitions).\\

The study of maximum sizes of induced trees in $G(n, p)$ has a deep history. It was initiated by Erd\H{o}s and Palka \cite{ERDOS1983145}. They showed that if $p=\op{const}$, then whp 
$$
 \mathbf{T}:=\mathrm{T}(G(n,p))=(2+o(1))\log_{1/(1-p)}(np)
$$ 
and posed the problem of determining the size of a largest induced tree in sparse random graphs. More precisely, they conjectured that for every $C>1$ whp $G(n, C/n)$ contains an induced tree of linear size. The conjecture was solved by Fernandez de la Vega~\cite{delaVega1986}, Frieze and Jackson~\cite{FRIEZE1987181}, Ku{\v{c}}era and R\"{o}dl \cite{article}, and \L{}uczak and Palka \cite{LUCZAK1988257}. In \cite{Palka}, Palka and Ruci{\'{n}}ski proved bounds for $\mathbf{T}$ when $p=\frac{C \ln n }{n}$ and $C>e$. Later~\cite{delaVega1996} Fernandez de la Vega improved the result for $p=C/n$: for large constant $C$ and any fixed $\varepsilon>0$, whp
$$
\frac{2 n(\ln C-\ln \ln C-1)}{C}\le \mathbf{T}\le \frac{(1+\varepsilon)2n\ln C}{C}.
$$
His result can be generalised to all larger $p$ and so it gives tight asymptotics of $\mathbf{T}$ when $1/n\ll p\leq 1-\varepsilon$. Dutta and Subramanian \cite{Dutta2018} improved the error term for $n^{-1 / 2}(\ln n)^2\leq p\leq 1-\varepsilon$: they proved that whp  
$$
\mathbf{T}=2\left(\log_{1/(1-p)} (n p)\right)+O(1/p).
$$
Dragani{\'{c}}~\cite{Drag} proved  that, for any given tree $T$ with bounded maximum degree and of size at most $(2-\varepsilon)\log_{1/(1-p)}(np)$, whp $G(n,p)$ contains an induced copy of $T$  when $n^{-1/2}\ln^{10/9} n\leq  p\leq 0.99$.  Finally, for the case $p=\op{const}$, Kamaldinov, Skorkin, and the third author~\cite{Kamaldinov2021} proved 2-point concentration of $\mathbf{T}$.

The asymptotic behaviour of $\mathbf{F}:=\mathrm{F}(G(n,p))$ is also fairly well understood and it is similar to that of $\mathbf{T}$. In~\cite{Krivoshapko2021}, Krivoshapko and the third author proved 2-point concentration of $\mathbf{F}$ for all constant $p\in(0,1)$. In~\cite{AkhK}, the first two authors proved that $\mathbf{F}$ is concentrated in an interval of size $o(1/p)$ for all $1/n\ll p=o(1)$: for any fixed $\varepsilon>0$ there exists $C=C(\varepsilon)$ such that, for all $C/n\leq p\leq 1-\varepsilon$, whp
$$
 \lfloor 2\log_{1/(1-p)}(enp(1-\varepsilon))+3\rfloor\leq\mathbf{F}\leq\lceil 2\log_{1/(1-p)}(enp(1+\varepsilon))+3\rceil.
$$
Glock~\cite{Glock2021} proved  that, for any given forest $F$ with bounded maximum degree and of size at most $(2-\varepsilon)\log_{1/(1-p)}(np)$, whp $G(n,p)$ contains an induced copy of $F$  when $p\gg n^{-1/2}\log n$. Some partial results have been obtained for forests with bounded maximum degrees: the asymptotic behaviour of maximum sizes of an induced matching~\cite{Cooley2021}, an induced forest whose components are all isomorphic to an arbitrary fixed tree \cite{Dragani2022}, and an induced forest with maximum degree at most 2~\cite{Dragani2022}.

It is harder to analyse trees and forests of bounded degrees due to the lack of counting results. Namely, the computation of first moments of the respective random variables requires an enumeration of labelled trees and weighted forests with bounded maximum degrees. A straightforward estimation of second moments leads to an estimation of the number of labelled trees and forests containing a fixed induced subforest. In this paper, we obtain the required enumeration results and use them to prove 2-point concentration of $\mathbf{T}_{\Delta}$ and $\mathbf{F}_{\Delta}$ when $p$ is constant, as well as asymptotics and concentration in an interval of size $o(1/p)$ of $\mathbf{F}_{\Delta}$ for all $1/n\ll p=o(1)$.\\

The following function plays a crucial role in our analysis. For every integer $\Delta\geq 2$, we define
\begin{equation}
\gamma_{\Delta}(x):=\sum_{k=0}^{\Delta-1}\frac{x^k}{k!},\quad x>0.
\label{eq:f_Delta_definition}
\end{equation}
For $\Delta\geq 3$, we let $\alpha_{\Delta}\in(1,\infty)$ be the unique solution of the equation
\begin{equation}
 \frac{x\gamma_{\Delta-1}(x)}{\gamma_{\Delta}(x)}=1
 \label{eq:alpha_Delta_definition}
\end{equation}
and $a_{\Delta}:=\gamma_{\Delta-1}(\alpha_{\Delta})$. We note that the solution indeed exists since $x\gamma_{\Delta-1}(x)-\gamma_{\Delta}(x)$ strictly increases on $[1,\infty)$, tends to infinity as $x\to\infty$, and equals $\gamma_{\Delta-1}(1)-\gamma_{\Delta}(1)=-1/(\Delta-1)!<0$ at $x=1$. In particular, $\alpha_3=\sqrt{2}$ and $a_3=\gamma_2(\alpha_3)=1+\sqrt{2}$. In Section~\ref{sc:auxiliary}, we prove the following important property of this sequence:
\begin{lemma}
The sequence $a_{\Delta},$ $\Delta\geq 3$, is non-decreasing and $\lim_{\Delta\to\infty}a_{\Delta}=e$.
\label{lm:a_Delta}
\end{lemma}

The rest of the introduction is organised as follows. In Section~\ref{sc:intro_dense} we give a brief overview of known results on maximum sizes of induced subgraphs in dense binomial random graphs and present our results on 2-point concentration of $\mathbf{T}_{\Delta}$ and $\mathbf{F}_{\Delta}$. In Section~\ref{sc:intro_sparse} we consider sparse settings $p=o(1)$, review known results, and present our result on asymptotics and concentration of $\mathbf{F}_{\Delta}$. Finally, Section~\ref{sc:intro_enumeration} is devoted to our proof methods and, in particular, to new enumeration results that can be on their own interest.


\subsection{Dense random graphs}
\label{sc:intro_dense}

The first result in the context of maximum sizes of induced subgraph in $G(n,p)$ was devoted to the maximum size of an empty graph (with no edges) --- i.e. to the {\it independence number} $\alpha(G(n,p))$. Matula~\cite{M1,M2,M3}, Grimmett and McDiarmid~\cite{GM}, and Bollob\'{a}s and Erd\H{o}s~\cite{Bollobas1976}
 proved that, for every constant $p\in(0,1)$, there exists a function $f(n)\sim 2\log_{1/(1-p)}(np)$ such that whp
$$
f(n)\leq\alpha(G(n,p))\leq f(n)+1.
$$
Since then, a series of results have demonstrated that this {\it 2-point concentration} in {\it dense} random graph $G(n,p=\mathrm{const})$ also holds for induced subgraphs from many other families: in particular, for the maximum size of
\begin{itemize}
\item an induced subgraph of bounded maximum degree \cite{Fountoulakis2010},
\item an induced subgraph of bounded average degree \cite{Fountoulakis2014},
\item an induced subgraph of given average degree \cite{Kamaldinov2021},
\item induced path and cycle \cite{Dutta2018},
\item an induced tree \cite{Kamaldinov2021}, and
\item an induced forest \cite{Krivoshapko2021}.
\end{itemize}
 
In this paper, we prove 2-point concentration for trees of bounded degree and forests of bounded degree.
\begin{theorem}\label{th:concentration_max_ind_tree_bounded_deg}
 Let $p\in(0,1)$ be constant and set $q:=1/(1-p)$. Then, for any constant $\Delta\ge3$, whp
\begin{enumerate}
\item $\mathbf{T}_{\Delta}$
equals either $\ceil{2\log_{q}\pr{a_{\Delta}np}+1}$ or $\ceil{2\log_{q}\pr{a_{\Delta}np}+2}$;
\item $\mathbf{F}_{\Delta}$
equals either $\ceil{2\log_{q}\pr{a_{\Delta}np}+1}$ or $\ceil{2\log_{q}\pr{a_{\Delta}np}+2}$.
\end{enumerate}
\end{theorem}

We actually prove a slightly stronger lower bound: for every fixed $\varepsilon>0$, whp $\mathbf{T}_{\Delta}\geq\lceil2\log_{1/(1-p)}(a_{\Delta}np)+3-\varepsilon\rceil$ (see Section~\ref{sub_indiced_subtrees}). This, together with Lemma~\ref{lm:a_Delta} and with the upper bound $\mathbf{F} \le \ceil{2\log_{1/(1-p)}\pr{enp}+2}$ whp  (see~\cite{Krivoshapko2021}\footnote{In~\cite{Krivoshapko2021}, a slightly weaker inequality  $\mathbf{F} \le \lfloor 2\log_{q}\pr{enp}+3+\varepsilon\rfloor$ is proven. Neverthess, exactly the same argument gives the stronger bound as well. Alternatively, computations from Section~\ref{sc:const_p} can be similarly applied to the number of forests without degree restrictions giving the desired bound.}) immediately imply that, for any function $\Delta(n)$ approaching infinity, both $\mathbf{T}_{\Delta(n)}$ and $\mathbf{F}_{\Delta(n)}$ are also concentrated in at most $2$ values.

\begin{corollary}\label{corollary}
Let $p\in(0,1)$ be constant and set $q:=1/(1-p)$. Then, for any $\Delta(n)\to\infty$, whp
$$
\ceil{2\log_{q}\pr{enp}+1} \le \mathbf{T}_{\Delta(n)} \le \mathbf{F}_{\Delta(n)} \le \ceil{2\log_{q}\pr{enp}+2}.
$$
\end{corollary}

We would also like to mention a striking result by Bollob{\'{a}}s and Thomason proving asymptotics of the maximum size of an induced subgraph with {\bf any} given hereditary property in dense random graphs \cite{Bollobs2000} as well as the fact that 2-point concentration fails in certain cases when typical maximum sizes are large~\cite{Balogh}.  

\subsection{Sparse random graphs}
\label{sc:intro_sparse}

In {\it sparse} random graphs $G(n,p=o(1))$ the maximum size of an induced subgraph from some of the above mentioned families is still concentrated around the value $2\log_{1/(1-q)}(np)$, although not necessarily in two points. In particular, Frieze~\cite{Frieze1990} proved that there exists a function $f(n)\sim2\log_{1/(1-p)}(np)$ such that for any $\varepsilon>0$, if $C/n<p=o(1)$, where $C=C(\varepsilon)$ is a sufficiently large constant, then whp
$$
|\alpha(G(n, p))-f(n)| \leq \frac{\varepsilon}{p}.
$$
In particular, the concentration interval has length $o(1/p)$ when $1/n\ll p=o(1)$. Noteworthy, Bohman and Hofstad~\cite{BH} proved two point concentration of $\alpha(G(n,p))$ for $p\geq n^{-2/3+\varepsilon}$ and noted that Sah and Sawney observed that this bound on $p$ is roughly best possible: if $1/n\ll p\ll(\ln n/n)^{2/3}$, then $\alpha(G(n,p))$ is not concentrated in any interval of a bounded size with probability $\Omega(1)$.

Furthermore, asymptotics of 
 the maximum sizes of 
\begin{itemize}
\item an induced tree \cite{Dutta2018,delaVega1996} (in \cite{delaVega1996} proof is provided for $p=C/n$ but works for larger $p$ as well),
\item an induced matching~\cite{Cooley2021},
\item induced path and cycle \cite{Dutta2018,Dragani2022},
\item an induced forest~\cite{AkhK},
\item an induced forest whose components are all isomorphic to an arbitrary fixed tree \cite{Dragani2022}
\end{itemize}
are known for $1/n\ll p=o(1)$, though, for the entire range of $p$, a concentration interval of length $o(1/p)$ is proven only for $\alpha(G(n,p))$ and $\mathbf{F}$.


It is reasonable to expect concentration in an interval of size $o(1/p)$ at least for most of the above mentioned graph classes, for which asymptotics are known. In the present paper we prove this for forests of bounded degree:

\begin{theorem}\label{th:concentration_max_ind_rooted_forests_bounded_degree}
 For any constant $\Delta\ge3$ and any $\varepsilon>0$, there exists $C$ such that, if $C/n\leq p\leq 1-\varepsilon$, then whp
$$
\floor{2\log_q(a_{\Delta}np(1-\varepsilon))+3}\le \mathbf{F}_{\Delta}\le \ceil{2\log_q(a_{\Delta}np(1+\varepsilon))+3},
$$
where $q=1/(1-p)$.
\end{theorem}

Theorem~\ref{th:concentration_max_ind_rooted_forests_bounded_degree} and Lemma~\ref{lm:a_Delta} together with the upper bound $\mathbf{F} \le \ceil{2\log_{1/(1-p)}\pr{enp(1+\varepsilon)}+3}$ whp~\cite{AkhK} immediately imply that, for any function $\Delta(n)$ approaching infinity and any $1/n\ll p=o(1)$, $\mathbf{F}_{\Delta(n)}$ is also concentrated in an interval of length $o(1/p)$.

\begin{corollary}\label{th:concentration_max_ind_rooted_forests}
For any fixed $\varepsilon>0$ and any $\Delta(n)\to\infty$, there exists $C=C(\varepsilon)$ such that, if $C/n\leq p\leq 1-\varepsilon$, then whp
$$
\floor{2\log_q(enp(1-\varepsilon))+3}\le \mathbf{F}_{\Delta(n)} \le \mathbf{F} \le \ceil{2\log_q(enp(1+\varepsilon))+3},
$$
where $q=1/(1-p)$.
\end{corollary}

Note that, in the extreme case of $p=\op{const}$, Theorem~\ref{th:concentration_max_ind_rooted_forests_bounded_degree} and Corollary~\ref{th:concentration_max_ind_rooted_forests} give concentration in three consecutive points. 

Unfortunately, our attempts to obtain a similar result for trees were unsuccessful. Nevertheless, we conjecture that such a result holds.
\begin{conjecture}
For any constant $\Delta\ge3$ and any $\varepsilon>0$, there exists $C$ such that, if $C/n\leq p\leq 1-\varepsilon$, then whp
$$
\floor{2\log_q(a_{\Delta}np(1-\varepsilon))+3}\le \mathbf{T}_{\Delta}\le \ceil{2\log_q(a_{\Delta}np(1+\varepsilon))+3},
$$
where $q=1/(1-p)$.
\end{conjecture}

\subsection{Counting trees and forests} 
\label{sc:intro_enumeration}

In our proofs, computing the first and second moments requires estimating the number of trees and forests with (1) bounded degrees and (2) a fixed subforest. Since these results could be useful on their own, we dedicate this section to them. We recall that the sequence of functions $\gamma_{\Delta}$ is defined in~\eqref{eq:f_Delta_definition}.\\

We prove the following generalisation of Cayley's formula. A {\it rooted forest} is a vertex-disjoint union of rooted trees, i.e. each component has a single root.
\begin{theorem}\label{lemma_tree_of_bounded_degree}
 For any fixed integer $\Delta\ge3$ and any $0<w=w(n)\leq n$,
\begin{enumerate}
\item the number of labelled trees on $[n]$ with maximum degree at most $\Delta$ equals
$$
t_{\Delta}(n)=(1+o(1))\alpha_{\Delta}\sqrt{\frac{\gamma_{\Delta}(\alpha_{\Delta})}{\gamma_{\Delta-2}(\alpha_{\Delta})}}\pr{\frac{a_{\Delta}}{e}}^n n^{n-2};
$$

\item letting $f_{\Delta}(n,m)$ be the number of labelled rooted forests on $[n]$ with maximum degree at most $\Delta$, with  $m$ components, and roots having degrees at most $\Delta-1$,
$$
\sum_{m=1}^n f_{\Delta}(n,m) w^m = (1+o(1))w\frac{(n-1)!}{r^{n-1}}\frac{(\gamma_{\Delta}(r))^{n}e^{r w}}{\sqrt{2\pi\beta(r) n}},
$$
where $r$ is the unique positive solution of the equation $\gamma'_{\Delta}=0$ and 
\begin{equation}
\beta(r)=\frac{r^2\gamma''_{\Delta}(r)}{\gamma_{\Delta}(r)}+\frac{2rw}{n}-\frac{r^2w^2}{n^2}.
\label{eq:beta_definition}
\end{equation}
\end{enumerate}
\end{theorem}


We prove the first part of Theorem~\ref{lemma_tree_of_bounded_degree} using the method of generating functions and a neat probabilistic approach: we express the number of labelled trees in terms of the probability that a certain random variable equals $n-2$. The later probability is estimated via the local limit theorem. The proof of the second part also relies on generating functions but is more technically involved: using Cauchy's coefficient formula, we express the estimated quantity in terms of a contour integral, and then approximate the latter using the saddle-point method.\\ 



Proving sufficient upper bounds for second moments of the desired random variables also requires enumeration of trees and forests containing a fixed induced forest.

\begin{lemma}\label{lemma_labelled_trees_with_a_given_subgraph}
Let $F$ be an acyclic graph on $[n]$ with $m$ non-empty connected components of sizes $\ell_1, \ldots, \ell_m$, let $\ell=\ell_1+\ldots+\ell_m$. Then,
\begin{enumerate}
\item the number of trees on $[n]$ containing $F$ as an induced subgraph equals 
$$
t(n,F)=\ell_1\ldots \ell_m n^{n-\ell-1}(n-\ell)^{m-1};
$$
\item the number of forests on $[n]$ with $h$ components containing $F$ as an induced subgraph equals 
$$
f(n,h,F)=\ell_1\ldots \ell_m
\sum\limits_{j=0}^{n-\ell} \binom{n-\ell}{j}\ell^{n-\ell-j}\binom{j+m-1}{h-1}(n-\ell)^{j+m-h}.
$$
\end{enumerate}
\end{lemma}

The proof of Lemma~\ref{lemma_labelled_trees_with_a_given_subgraph} is relatively standard: we use a modification of the Pr\"{u}fer sequence to encode trees with a fixed independent set and then present an algorithm of reducing the initial problem to encoding such trees.\\

The rest of the paper is organised as follows. We prove Theorem~\ref{lemma_tree_of_bounded_degree} and Lemma~\ref{lemma_labelled_trees_with_a_given_subgraph} in Section~\ref{sc:counting}. Then, in Section~\ref{sc:auxiliary}, we use these two results to get some auxiliary bounds. Theorems~\ref{th:concentration_max_ind_tree_bounded_deg}~and~\ref{th:concentration_max_ind_rooted_forests_bounded_degree} are proven in Sections~\ref{sc:const_p}~and~\ref{sc:small_p} respectively. 

\section{Counting trees and forests of bounded degree}
\label{sc:counting}

\subsection{Proof of part 1 of Theorem~\ref{lemma_tree_of_bounded_degree}}

For $\Delta\ge3$, we denote by $t_{\Delta}(n)$ the number trees on $[n]$ with maximum degree at most $\Delta$. In the usual way, we define the coefficient extraction operator $[x^n]$ applied to a power series $\sum_{k=0}^{\infty}c_k{x^k}$ as 
$$
[x^n]\sum\limits_{k=0}^{\infty}c_k{x^k}:=c_n.
$$
Recall that $\gamma_{\Delta}$ is defined in~\eqref{eq:f_Delta_definition}. Since 
$$
[x^{n-2}]\left((n-2)!\sum_{k=0}^{\Delta-1}{\frac{x^k}{k!}}\right)^n=\sum_{k_1,\ldots,k_n\leq\Delta-1:\,k_1+\ldots+k_n=n-2}\frac{(n-2)!}{k_1!\ldots k_n!},
$$
which is exactly the number of sequences in $[n]^{n-2}$ with each number from $[n]$ appearing at most $\Delta-1$ times, we get 
\begin{equation}
t_{\Delta}(n)=(n-2)![x^{n-2}]\pr{\gamma_{\Delta}(x)}^n=\left.\frac{d^{n-2}}{dx^{n-2}}(\gamma_{\Delta}(x))^n\right\vert_{x=0},
\label{eq:f->t}
\end{equation}
due to the encoding of trees with maximum degrees bounded by $\Delta$ via such Pr\"{u}fer sequences. 

Recalling that $\alpha_{\Delta}$ is a solution of the equation~\eqref{eq:alpha_Delta_definition} (see the definition in Section~\ref{sc:intro}), we consider a random variable $\xi$ with generating function 
 $$
 \varphi(x):=\mathbb{E}{x^{\xi}}=\frac{\gamma_{\Delta}(\alpha_{\Delta}x)}{\gamma_{\Delta}(\alpha_{\Delta})}=
 \sum_{k=0}^{\Delta-1}\frac{\alpha_{\Delta}^k}{k!\gamma_{\Delta}(\alpha_{\Delta})}x^k.
 $$
In other words, for every $k\in[0,\ldots,\Delta-1]$, $\mathbb{P}(\xi=k)=\frac{\alpha_{\Delta}^k}{k!\gamma_{\Delta}(\alpha_{\Delta})}$. Now, let $\xi_1,\ldots,\xi_n$ be independent copies of $\xi$ and $S_n=\xi_1+\ldots+\xi_n$. Clearly, the generating function of $S_n$ equals
$$
\varphi_{S_n}(x)=\pr{\varphi(x)}^n.
$$
Then, for $m\in\mathbb{Z}_{\geq 0}$:
\begin{align*}
\Pb{S_n=m}=[x^m]\varphi_{S_n}(x) &=
\left.\frac{1}{m!}\frac{d^m}{dx^m}\varphi_{S_n}(x)\right\vert_{x=0}\\
&=
\left.\frac{\alpha_{\Delta}^{m}}{m!}\frac{d^m}{dx^m}\varphi_{S_n}\pr{\frac{x}{\alpha_{\Delta}}}\right\vert_{x=0}=
\frac{\alpha_{\Delta}^{m}}{m!\pr{\gamma_{\Delta}(\alpha_{\Delta})}^n}\left.\frac{d^m}{dx^m}\pr{\gamma_{\Delta}(x)}^n\right\vert_{x=0}.
\end{align*}
Due to~\eqref{eq:f->t}, setting $m=n-2$, we get
\begin{equation}
t_{\Delta}(n)=\frac{\pr{\gamma_{\Delta}(\alpha_{\Delta})}^n}{\alpha_{\Delta}^{n-2}}(n-2)!\cdot\Pb{S_n=n-2}.
\label{eq:t_delta_through_prob}
\end{equation}
By the definition of $\alpha_{\Delta}$ and since $\gamma'_{\Delta}(x)=\gamma_{\Delta-1}(x)$, we get
$$
\mathbb{E} S_n=n\mathbb{E} \xi =n\varphi'(1)=n\frac{\alpha_{\Delta}\gamma_{\Delta-1}(\alpha_{\Delta})}{\gamma_{\Delta}(\alpha_{\Delta})}=n
$$
and
$$
\mathrm{Var}S_n=n\mathrm{Var}\xi=n(\varphi''(1)-\mathbb{E}\xi(\mathbb{E}\xi-1))=n\varphi''(1)=n\frac{\alpha_{\Delta}^{2}\gamma_{\Delta-2}(\alpha_{\Delta})}{\gamma_{\Delta}(\alpha_{\Delta})}.
$$
Applying the local limit theorem (see, e.g.,~\cite[Chapter 9]{GnedenkoKolmogorov}), we get
$$
\lim\limits_{n\to+\infty}\left|\sqrt{\mathrm{Var}{S_n}}\cdot\Pb{S_n=n-2}-\frac{1}{\sqrt{2\pi}}e^{-\frac{\pr{n-2-\mathbb{E}{S_n}}^2}{2\mathrm{Var}{S_n}}}\right|=0.
$$
Since $\mathrm{Var}{S_n}\to\infty$ as $n\to\infty$, we get that
$$
\lim\limits_{n\to\infty}\sqrt{n\mathrm{Var}\xi}\cdot\Pb{S_n=n-2}=\frac{1}{\sqrt{2\pi}}.
$$
Finally, recalling that $a_{\Delta}=\gamma_{\Delta-1}(\alpha_{\Delta})$, from~\eqref{eq:t_delta_through_prob} and since $x=\alpha_{\Delta}$ satisfies~\eqref{eq:alpha_Delta_definition},
\begin{align}
t_{\Delta}(n)=(1+o(1))\frac{\pr{\gamma_{\Delta}(\alpha_{\Delta})}^n}{\alpha_{\Delta}^{n-2}}\frac{(n-2)!}{\sqrt{2\pi{n}\D{\xi_{\Delta}}}}
 &=(1+o(1))\frac{\alpha_{\Delta}\pr{\gamma_{\Delta-1}(\alpha_{\Delta})}^n}{\sqrt{2\pi{n}\frac{\gamma_{\Delta-2}(\alpha_{\Delta})}{\gamma_{\Delta}(\alpha_{\Delta})}}}(n-2)!\notag\\
 &=(1+o(1))\alpha_{\Delta}\sqrt{\frac{\gamma_{\Delta}(\alpha_{\Delta})}{\gamma_{\Delta-2}(\alpha_{\Delta})}}\pr{\frac{a_{\Delta}}{e}}^n{n^{n-2}},
 \label{eq:t_Delta_final_asymp}
\end{align}
completing the proof.\\


\begin{remark}
A similar method was used in~\cite{Britikov1988} to estimate the number of labelled forests with a fixed number of components.
\end{remark}

\subsection{Proof of part 2 of Theorem~\ref{lemma_tree_of_bounded_degree}}
\label{sc:part2_proof}

Rooted forests with $m$ components can be encoded via Pr\"{u}fer sequences in the following way. Without loss of generality assume that the roots of a given forest on $[n]$ comprise the set $[m]$. The Pr\"{u}fer sequence is constructed using the standard procedure, pruning non-root leaves with the largest labels first until only roots survive. The sequence takes a form $(a_1,\ldots,a_{n-1-m},b)$, with $a_i\in[n], b\in[m]$ since the last element of the code can only be a root. 

Recall that the number of occurrences of a vertex in a Pr\"{u}fer sequence $(a_1,\ldots,a_{n-1-m})$ equals its degree minus one unless this vertex is a root. The number of occurrences of a root is at most its degree unless this root is $b$, and the number of occurrences of $b$ is at most its degree minus one, as for the non-root vertices. Hence, the following factors contribute to $f_{\Delta}(n, m)$: 
\begin{itemize}
\item the number of choices of the set of roots, which is ${n\choose m}$;
\item the number of choices of $b$, which is $m$;
\item the number of Pr\"{u}fer sequences $(a_1,\ldots,a_{n-1-m})$ where each number from $[n]\setminus\{b\}$ appears in $(a_1,\ldots,a_{n-1-m})$ at most $\Delta-1$ times, while $b$ appears at most $\Delta-2$ times  (it is clear that the defined map from rooted forests to sequences is a bijection). 
\end{itemize}
Thus,
\begin{align*}
f_{\Delta}(n, m) & =\binom{n}{m}m(n-m-1)![x^{n-m-1}]\left(\sum_{j=0}^{\Delta-2}\frac{x^j}{j!}\right)\left(\sum_{j=0}^{\Delta-1}\frac{x^j}{j!}\right)^{n-1}\\
&=
\binom{n}{m}m(n-m-1)![x^{n-m-1}]\gamma_{\Delta-1}(x)(\gamma_{\Delta}(x))^{n-1}\\
&=\binom{n}{m}\cdot\frac{m}{n}\cdot(n-m-1)![x^{n-m-1}]\frac{\partial}{\partial x}(\gamma_{\Delta}(x))^{n}\\
&=\binom{n-1}{m-1}(n-m)![x^{n-m}](\gamma_{\Delta}(x))^{n}
=\frac{(n-1)!}{(m-1)!}[x^{n}]x^{m}(\gamma_{\Delta}(x))^{n}.
\end{align*}
Then,
\begin{align}
\sum\limits_{m=1}^{n}f_{\Delta}(n,m)w^m
&=\sum\limits_{m=1}^{n}\frac{(n-1)!}{(m-1)!}[x^{n}]x^{m}\pr{\gamma_{\Delta}(x)}^n w^m\notag\\
&=(n-1)!w[x^{n}]x\pr{\gamma_{\Delta}(x)}^n\sum\limits_{j=0}^{n-1}\frac{w^j x^j}{j!}\notag\\
&=(n-1)!w[x^{n}]x\pr{\gamma_{\Delta}(x)}^n\sum\limits_{j=0}^{\infty}\frac{w^j x^j}{j!}
=(n-1)!w[x^{n}]x\pr{\gamma_{\Delta}(x)}^ne^{wx}.
\label{eq:forests_weighted_generating}
\end{align}
By the Cauchy's coefficient formula (see, e.g.,~\cite[Theorem IV.4]{Flajolet2009}),
$$
[x^{n}]x\pr{\gamma_{\Delta}(x)}^{n}e^{wx}
=\frac{1}{2\pi i}\oint z\pr{\gamma_{\Delta}(z)}^{n}e^{wz}\frac{dz}{z^{n+1}}
=\frac{1}{2\pi i}\oint ze^{g(z)}\frac{dz}{z},
$$
where 
\begin{equation}
g(z)=n\ln(\gamma_\Delta(z))-n\ln z + wz.
\label{eq:g_def}
\end{equation}
Now, let us choose $r=r(n)>0$ so that it is the saddle point, i.e. it satisfies \textit{the saddle-point equation} $g'(x)=0$. The saddle-point equation is equivalent to
$$
\frac{\partial}{\partial x}e^{g(x)}=0 \,\Leftrightarrow\,
\frac{\partial}{\partial x}\pr{\frac{\gamma_{\Delta}(x)e^{\frac{wx}{n}}}{x}}^n=0\,\Leftrightarrow
\frac{\partial}{\partial x}\pr{\frac{\gamma_{\Delta}(x)e^{\frac{wx}{n}}}{x}}=0.
$$
Let $\tilde g(x)=\gamma_{\Delta}(x)\exp\{wx/n\}/x$. Since $\lim_{x\to 0+}\tilde g(x)=\infty$ and $\tilde g'(x)>0$ for large $x$, the equation has a positive root. Moreover, since the power series (at $x=0$) of $xg$ has non-negative coefficients, the solution is unique (see~\cite[VIII.4]{Flajolet2009}) . In what follows, we estimate the Cauchy coefficient integral using the classical saddle-point method, see details in~\cite[Chapter VIII]{Flajolet2009}.

Integrating along the circle of radius $r$ and using polar coordinates, i.e. setting $z=re^{i\theta}$,
$$
[x^{n}]x\pr{\gamma_{\Delta}(x)}^{n}e^{wn}
=\frac{1}{2\pi}\int\limits_{-\pi}^{\pi}re^{i\theta}e^{g(re^{i\theta})}d\theta.
$$
Next, let us show that $r=O(1)$ and $r<n/w$. Let us rewrite the saddle-point equation as follows:
\begin{equation}
\frac{x\gamma'_{\Delta}(x)}{\gamma_{\Delta}(x)}
+\frac{xw}{n}=1.
\label{eq:saddle_point}
\end{equation}
Since
$$
\left.\pr{\frac{x\gamma'_{\Delta}(x)}{\gamma_{\Delta}(x)}+\frac{xw}{n}}\right|_{x=0}=0,\,\,\,
\left.\pr{\frac{x\gamma'_{\Delta}(x)}{\gamma_{\Delta}(x)}+\frac{xw}{n}}\right|_{x=\frac{n}{w}}>1,\,\,\,
\left.\pr{\frac{x\gamma'_{\Delta}(x)}{\gamma_{\Delta}(x)}+\frac{xw}{n}}\right|_{x=2}
>\frac{2\gamma'_{\Delta}(2)}{\gamma_{\Delta}(2)}>1,
$$
we immediately get
\begin{claim}
$r\in(0,\min\{2,n/w\})$.
\label{cl:cl1}
\end{claim}

Set $\theta_0=n^{-2/5}$. For $|\theta|\le\theta_0$:
\begin{align}
g(re^{i\theta})-g(r)
&=n\ln\pr{\frac{\gamma_{\Delta}(re^{i\theta})}{\gamma_{\Delta}(r)e^{i\theta}}}
+wr(e^{i\theta}-1)\notag\\
&\stackrel{(*)}=n\ln\pr{\frac{\gamma_{\Delta}(r)+\gamma'_{\Delta}(r)r(e^{i\theta}-1) +\frac{1}{2}\gamma''_{\Delta}(r)r^2(e^{i\theta}-1)^2+O(\theta^3)}{\gamma_{\Delta}(r)e^{i\theta}}}
+wr(e^{i\theta}-1)\notag\\
&\stackrel{\eqref{eq:saddle_point}}=n\ln\pr{\frac{\gamma_{\Delta}(r)e^{i\theta}
-r\gamma_{\Delta}(r)\frac{e^{i\theta}-1}{n/w}
-\frac{r^2}{2}\gamma''_{\Delta}(r)\theta^2+O(\theta^3)}{\gamma_{\Delta}(r)e^{i\theta}}}
+wr(e^{i\theta}-1)\notag\\
&=n\ln\pr{1
-\frac{wr(i\theta+\frac{\theta^2}{2})}{n}
-\frac{\frac{r^2}{2}\gamma''_{\Delta}(r)\theta^2}{\gamma_{\Delta}(r)}+O(\theta^3)}
+wr\left(i\theta-\frac{\theta^2}{2}+O(\theta^3)\right)\notag\\
&=-wr\left(i\theta+\frac{\theta^2}{2}\right)
+\frac{n}{2}\frac{w^2r^2\theta^2}{n^2}
-\frac{r^2\gamma''_{\Delta}(r)}{2\gamma_{\Delta}(r)}n\theta^2
+O(n\theta^3)
+wr\left(i\theta-\frac{\theta^2}{2}\right)\notag\\
&=-\pr{\frac{r^2\gamma''_{\Delta}(r)}{2\gamma_{\Delta}(r)}+\frac{rw}{n}-\frac{r^2w^2}{2n^2}}n\theta^2
+O(n\theta^3)\notag\\
&=-\frac{1}{2}\beta(r)n\theta^2+O(n\theta^3)=-\frac{1}{2}\beta(r)n\theta^2+o(1),
\label{eq:g-g-beta}
\end{align}
where $\beta(r)$ is defined in~\eqref{eq:beta_definition} and the equality (*) holds since $\gamma_{\Delta}$ is analytic.

\begin{claim}
$\beta(r)=\Theta(1)$.
\label{cl:cl2}
\end{claim}
\begin{proof}
First, $\beta(r)>0$ since $r<n/w$. Secondly, $\beta(r)=O(1)$ since $r=O(1)$ due to Claim~\ref{cl:cl1}. Finally, assume that $\beta(r)\to0$. Then, necessarily, $\frac{r^2\gamma''_{\Delta}(r)}{2\gamma_{\Delta}(r)}\to0$ and, 
therefore, $r\to0$. However, from~\eqref{eq:saddle_point}, in this case $r=(1+o(1))n/w$ implying that $\beta(r)=1+o(1)$, which contradicts the assumption that $\beta(r)\to0$. Thus, $\beta(r)=\Theta(1)$.
\end{proof}
We finally get
\begin{equation}
\frac{1}{2\pi}\int\limits_{-\theta_0}^{\theta_0}re^{i\theta}e^{g(re^{i\theta})}d\theta
=\frac{re^{g(r)+o(1)}}{2\pi}\int\limits_{-\theta_0}^{\theta_0}e^{-\frac{1}{2}\beta(r)n\theta^2}d\theta.
\label{eq:integral_part1}
\end{equation}

We now consider $\theta_0\le|\theta|\le\pi$. Let us show that $|e^{g(re^{i\theta})}|\le |e^{g(re^{i\theta_0})}|$ in this case. We have
\begin{equation}
\left|e^{g(re^{i\theta})}\right|
=\left|\frac{\gamma_{\Delta}(re^{i\theta})}{r}\right|^{n}
\cdot\left|e^{e^{i\theta}}\right|^{wr}
=\left|\frac{\gamma_{\Delta}(re^{i\theta})}{r}\right|^{n}
\cdot\left(e^{\cos\theta}\right)^{wr}\le \left|\frac{\gamma_{\Delta}(re^{i\theta})}{r}\right|^{n}\cdot\left|e^{e^{i\theta_0}}\right|^{wr}
\label{eq:g_abs}
\end{equation}
since $e^{\cos\theta}\le e^{\cos\theta_0}=|e^{e^{i\theta_0}}|$.
Next,
$$
|\gamma_{\Delta}(re^{i\theta})|^2
=\pr{\sum\limits_{s=0}^{\Delta-1}\frac{r^{s}e^{is\theta}}{s!}}
\pr{\sum\limits_{h=0}^{\Delta-1}\frac{r^{h}e^{-ih\theta}}{h!}}
=\sum\limits_{s,h=0}^{\Delta-1}\frac{r^{s+h}}{s!h!}\cos((s-h)\theta),
$$
thus,
\begin{align}
\left|\gamma_{\Delta}(re^{i\theta_0})\right|^2-\left|\gamma_{\Delta}(re^{i\theta})\right|^2
&=2\sum\limits_{s,h=0}^{\Delta-1}\frac{r^{s+h}}{s!h!}\sin\pr{(s-h)\frac{\theta+\theta_0}{2}}\sin\pr{(s-h)\frac{\theta-\theta_0}{2}}\notag\\
&=2\sum\limits_{s,h=0}^{\Delta-1}\frac{r^{s+h}}{s!h!}\sin\pr{|s-h|\frac{|\theta|+\theta_0}{2}}\sin\pr{|s-h|\frac{|\theta|-\theta_0}{2}}.
\label{eq:sinus}
\end{align}
Notice that in the above sum a summand can be negative only if, for some $m\in\mathbb{Z}$,
$$
|s-h|\frac{|\theta|-\theta_0}{2}
<\pi m <
|s-h|\frac{|\theta|+\theta_0}{2},
$$
in which case
$$
\pi m - \Delta\theta_0 <
\pi m - |s-h|\theta_0 <
|s-h|\frac{|\theta|-\theta_0}{2}<
|s-h|\frac{|\theta|+\theta_0}{2}<
\pi m + |s-h|\theta_0<
\pi m + \Delta\theta_0
$$
meaning that
$$
\left|\sin\pr{|s-h|\frac{|\theta|\pm\theta_0}{2}}\right|
\le\Delta\theta_0.
$$
Therefore,
$$
|\gamma_{\Delta}(re^{i\theta_0})|^2-|\gamma_{\Delta}(re^{i\theta})|^2
\ge 4r\sin\pr{\frac{|\theta|+\theta_0}{2}}\sin\pr{\frac{|\theta|-\theta_0}{2}}
-\sum\limits_{s,h=1}^{\Delta-1}\frac{r^{s+h}}{s!h!}\Delta^2\theta_0^2,
$$
which implies, for $\sqrt{\theta_0}<|\theta|\le\pi$, that
$$
\left|\gamma_{\Delta}(re^{i\theta_0})\right|^2-\left|\gamma_{\Delta}(re^{i\theta})\right|^2
\ge 4r\frac{\theta_0}{4}+O(r^2\theta_0^2)
=r\theta_0(1+O(r\theta_0))>0.
$$
While for $\theta_0\le|\theta|\le\sqrt{\theta_0}$,
$$
0\le|s-h|\frac{|\theta|\pm\theta_0}{2}
\le\Delta\sqrt{\theta_0}<\pi
$$
meaning that each term in the sum in the right-hand side of~\eqref{eq:sinus} is non-negative.
Thus, for $\theta_0\le|\theta|\le\pi$,
$$
\left|\gamma_{\Delta}(re^{i\theta_0})\right|^2-\left|\gamma_{\Delta}(re^{i\theta})\right|^2\ge0
$$
and, therefore, due to~\eqref{eq:g_abs} and~\eqref{eq:g-g-beta},
$$
\left|e^{g(re^{i\theta})}\right|
\le\left|e^{g(re^{i\theta_0})}\right|
=e^{g(r)-\frac{1}{2}\beta(r)n\theta_0^2+o(1)},
$$
which, in turn, imples
\begin{equation}
\left|\frac{1}{2\pi}\int\limits_{\theta_0}^{2\pi-\theta_0}re^{i\theta}e^{g(re^{i\theta})}d\theta\right|
\le\frac{r}{2\pi}\int\limits_{\theta_0}^{2\pi-\theta_0}\left|e^{g(re^{i\theta})}\right|d\theta
\le re^{g(r)-\frac{1}{2}\beta(r)n\theta_0^2+o(1)}
=re^{g(r)-\frac{1}{2}\beta(r)n^{1/5}+o(1)}.
\label{eq:integral_part2}
\end{equation}
Further, due to Claim~\ref{cl:cl2},
\begin{equation}
\frac{re^{g(r)}}{\pi}\int\limits_{\theta_0}^{\infty}e^{-\frac{1}{2}\beta(r)n\theta^2}d\theta
<\frac{re^{g(r)}}{\pi\theta_0}\int\limits_{\theta_0}^{\infty}\theta e^{-\frac{1}{2}\beta(r)n\theta^2}d\theta
=\frac{re^{g(r)+O(1)}}{n\theta_0}e^{-\frac{1}{2}\beta(r)n\theta_0^2}
<re^{g(r)-\frac{1}{2}\beta(r)n\theta_0^2}.
\label{eq:integral_infty}
\end{equation}
Finally,
\begin{align*}
[x^{n}]x\pr{\gamma_{\Delta}(x)}^{n}e^{xw}
&=\frac{1}{2\pi}\int\limits_{-\theta_0}^{\theta_0}re^{i\theta}e^{g(re^{i\theta})}d\theta
+\frac{1}{2\pi}\int\limits_{\theta_0}^{2\pi-\theta_0}re^{i\theta}e^{g(re^{i\theta})}d\theta\\
&\stackrel{\eqref{eq:integral_part1},\eqref{eq:integral_part2}}=\frac{re^{g(r)+o(1)}}{2\pi}\int\limits_{-\theta_0}^{\theta_0}e^{-\frac{1}{2}\beta(r)n\theta^2}d\theta
+O\left(re^{g(r)-\frac{1}{2}\beta(r)n^{1/5}}\right)\\
&\stackrel{\eqref{eq:integral_infty}}=\frac{re^{g(r)+o(1)}}{2\pi}\int\limits_{-\infty}^{+\infty}e^{-\frac{1}{2}\beta(r)n\theta^2}d\theta
+O\left(re^{g(r)-\frac{1}{2}\beta(r)n^{1/5}}\right)\\
&=\frac{re^{g(r)+o(1)}}{\sqrt{2\pi\beta(r)n}}
+O\left(re^{g(r)-\frac{1}{2}\beta(r)n^{1/5}}\right)\stackrel{\eqref{eq:g_def}}=\frac{(\gamma_{\Delta}(r))^{n}e^{wr+o(1)}}{r^{n-1}\sqrt{2\pi\beta(r)n}},
\end{align*}
completing the proof due to~\eqref{eq:forests_weighted_generating}.

\subsection{Proof of Lemma~\ref{lemma_labelled_trees_with_a_given_subgraph}}

Recall that a labelled tree on vertices $[n]$ can be encoded by a Pr\"{u}fer sequence from $[n]^{n-2}$. Let us fix some $m\in[n]$. Then the standard encoding procedure (see, e.g.,~\cite{Moon1970}) can be slightly modified to encode only trees with the independent set $[m]$ by two sequences $a\in\{m+1,\ldots,n\}^{m-1}$ and $b\in[n]^{n-m-1}$.
Indeed, take any such tree and successively prune its leaves, starting from leaves with smallest labels, as in the standard encoding procedure, until only one edge remains. If the pruned leaf has the label in $[m]$ and at least 2 elements remain in $[m]$, then write the label of its neighbour to the sequence $a$, otherwise, to the sequence $b$. We stop adding elements to $a$ when a single vertex remains in $[m]$. It is easy to see that this procedure gives bijection between the set of all trees with independent set $[m]$ and the set of codes. If we additionally fix the degree sequence $(d_1,\ldots,d_n)$,  then we require from codewords to contain each $i\in[n]$ exactly $d_i-1$ times. We stress that elements from $[m]$ belong to $b$ only. Thus, the number of such labelled trees on $[n]$ with no edges between the first $m$ vertices and with the degree sequence $\pr{d_1,\ldots,d_n}$ is exactly
\begin{multline}
t_{d_1,\ldots,d_n}(n;m)=\\
=\binom{n-m-1}{d_1-1,d_2-1,\ldots,d_m-1,n-1-\pr{d_1+\ldots+d_m}}
\binom{n+m-2-\pr{d_1+\ldots+d_m}}{d_{m+1}-1,\ldots,d_n-1}.
\label{eq_n_trees_with_independent_set_given_deg}
\end{multline}
Indeed, there are $\binom{n-m-1}{d_1-1,d_2-1,\ldots,d_m-1,n-1-\pr{d_1+\ldots+d_m}}$ ways to choose $d_i-1$ coordinates in $b$ that equal to $i$ for every $i\in[m]$. Then, it remains to choose $d_i-1$ coordinates in the concatenation $ab$ excluding the elements that belong to $[m]$ (this remaining codeword has length $n+m-2-\pr{d_1+\ldots+d_m})$, for every $i\in[n]\setminus[m]$.\\

Fix a forest $F$ on $[n]$ with $m$ non-empty components of sizes $\ell_1,\ldots,\ell_m$ and let $\ell=|V(F)|=\ell_1+\ldots+\ell_m$. 


\paragraph{Proof of the first part of Lemma~\ref{lemma_labelled_trees_with_a_given_subgraph}.} Any tree on $[n]$ containing $F$ as an induced subgraph can be uniquely constructed as follows:
\begin{itemize}
\item contract each component of $F$ into a single vertex;
\item label the contracted components by numbers from $[m]$ and label remaining vertices by numbers from $\{m+1,\ldots,n-\ell+m\}$;
\item construct a tree $T$ on $[n-\ell+m]$ with an independent set $[m]$; and
\item expand the vertices from $[m]$ (which are the contracted components) and choose ends of edges of $T$ in these expansions.
\end{itemize} 
Thus, applying \eqref{eq_n_trees_with_independent_set_given_deg}, we get
\begin{align}
t(n,F) &=\sum\limits_{d_1,\ldots,d_{n-\ell+m}=1}^{\infty}
\binom{n-\ell-1}{d_1-1,d_2-1,\ldots,d_m-1,n-\ell+m-1-\pr{d_1+\ldots+d_m}}\times\notag\\
&\quad\quad\quad\quad\quad\,\,\,\times\binom{n-\ell+2m-2-\pr{d_1+\ldots+d_m}}{d_{m+1}-1,\ldots,d_{n-\ell+m}-1}\ell_1^{d_1}\cdot\ldots\cdot\ell_m^{d_m}\notag\\
&=\sum\limits_{k_0,k_1,\ldots,k_m=0}^{\infty}
\binom{n-\ell-1}{k_0,k_1,\ldots,k_m}\ell_1^{k_1+1}\ldots\ell_m^{k_m+1}
\sum\limits_{k_{m+1},\ldots,k_{n-\ell+m}=0}^{\infty}
\binom{k_0+m-1}{k_{m+1},\ldots,k_{n-\ell+m}}\notag\\
&=\sum\limits_{k_0,k_1,\ldots,k_{m}=0}^{+\infty}
\binom{n-\ell-1}{k_0,k_1,\ldots,k_{m}}
\ell_1^{k_1+1}\ldots \ell_m^{k_m+1}(n-\ell)^{k_{0}+m-1}=\ell_1\ldots\ell_m n^{n-\ell-1}(n-\ell)^{m-1},
\label{f(k,l,r)}
\end{align}
where all summations are over valid sequences such that the respective multinomial coefficients are well-defined. This completes the proof of the first part of Lemma~\ref{lemma_labelled_trees_with_a_given_subgraph}.

\paragraph{Proof of the second part of Lemma~\ref{lemma_labelled_trees_with_a_given_subgraph}.}
Each rooted forest on $[n]$ with $h$ components containing $F$ as an induced subgraph can be uniquely constructed as follows:
\begin{itemize}
    \item contract each component of $F$ into a single vertex;
    \item label the contracted components by numbers from $[m]$ and label all the remaining vertices by numbers from $\{m+1,\ldots,n-\ell+m\}$;
    \item add one more vertex with label $n-\ell+m+1$;
    \item construct a tree $T$ on $[n-\ell+m+1]$ with $\deg(n-\ell+m+1)=h$;
    \item expand the vertices in $[m]$ and choose ends of edges of $T$ in these expansions;
    \item remove the vertex $n-\ell+m+1$.
\end{itemize}
Thus, applying \eqref{eq_n_trees_with_independent_set_given_deg},
\begin{align*}
f(n,h,F) & =\sum\limits_{d_1,\ldots,d_{n-\ell+m}=1}^{\infty}
\binom{n-\ell}{d_1-1,d_2-1,\ldots,d_m-1,n-\ell+m-\pr{d_1+\ldots+d_m}}\times\\
&\quad\quad\quad\quad\quad\,\,\,\times\binom{n-\ell+2m-1-\pr{d_1+\ldots+d_m}}{d_{m+1}-1,\ldots,d_{n-\ell+m}-1,h-1}\ell_1^{d_1}\ldots\ell_m^{d_m}\\
&=\sum\limits_{k_0,k_1,\ldots,k_{m}=0}^{\infty}
\binom{n-\ell}{k_0,k_1,\ldots,k_m}\ell_1^{k_1+1}\ldots\ell_m^{k_m+1}\times\\
&\quad\quad\quad\quad\quad\,\,\,\times
\sum\limits_{k_{m+1},\ldots,k_{n-\ell+m}=0}^{\infty}
\binom{k_0+m-1}{k_{m+1},\ldots,k_{n-\ell+m},h-1}\\
&=\ell_1\ldots\ell_m
\sum\limits_{k_0=0}^{n-\ell}
\binom{n-\ell}{k_0}\ell^{n-\ell-k_0}
\binom{k_{0}+m-1}{h-1}(n-\ell)^{k_{0}+m-h},
\end{align*}
where all summations are over valid sequences such that the respective multinomial coefficients are well-defined. This completes the proof of the second part of Lemma~\ref{lemma_labelled_trees_with_a_given_subgraph}.


\section{Auxiliary claims}
\label{sc:auxiliary}

\subsection{Proof of Lemma~\ref{lm:a_Delta}}

First, notice that
$$
\alpha_{\Delta}=\frac{\gamma_{\Delta}(\alpha_{\Delta})}{\gamma_{\Delta-1}(\alpha_{\Delta})}=1+\frac{\alpha_{\Delta}^{\Delta-1}}{(\Delta-1)!\sum_{k=0}^{\Delta-2}\alpha_{\Delta}^k/k!}<
1+\frac{\alpha_{\Delta}^{\Delta-1}}{(\Delta-1)!\alpha_{\Delta}^{\Delta-2}/(\Delta-2)!}=1+\frac{\alpha_{\Delta}}{\Delta-1},
$$
therefore, for all $\Delta\ge3$, $1<\alpha_{\Delta}<\frac{\Delta-1}{\Delta-2}\le2.$ Moreover, $\lim_{\Delta\to\infty}\alpha_{\Delta}=1$. Applying Tannery's theorem about interchanging of the limit and the summation operations, we get
$$
\lim\limits_{\Delta\to\infty}a_{\Delta}=
\lim\limits_{\Delta\to\infty}\sum\limits_{k=0}^{\Delta-2}\frac{\alpha_{\Delta}^k}{k!}=
\sum\limits_{k=0}^{\infty}\lim\limits_{\Delta\to\infty}\frac{\alpha_{\Delta}^k}{k!}=
\sum\limits_{k=0}^{\infty}\frac{1}{k!}=e.
$$
The fact that $a_{\Delta}$ is non-decreasing follows from \eqref{eq:t_Delta_final_asymp} and an obvious observationt that $t_{\Delta}(n)$ is non-decreasing in $\Delta$.

\subsection{Trees with a fixed subforest}

Fix a forest $F$ on $[n]$ with $m$ non-empty components of sizes $\ell_1,\ldots,\ell_m$ and let $\ell=|V(F)|=\ell_1+\ldots+\ell_m$. 

\begin{claim}
Let $T$ be a tree on $[n]$ with maximum degree $\Delta$ that contains $F$ as an induced subgraph. Then $m\leq 1+(n-\ell)(\Delta-1)$.
\label{cl:cl3}
\end{claim}

\begin{proof}
Let us consider an auxiliary tree $T'$ that is obtained from $T$ by contracting each component of $F$ into a single vertex. Without loss of generality, assume that $T'$ has $V(T')=[n-\ell+m]$ and that $[m]$ contains all contracted vertices. Let $(d_1,\ldots,d_{n-\ell+m})$ be the degree sequence of $T'$. Then $\sum_{j=1}^{n-\ell+m} d_j=2(n-1-\ell+m)$. Moreover, $d_1+\ldots+d_m\leq n-1-\ell+m$ since there are no edges between the vertices from $[m]$ in $T'$. Therefore, $\sum_{j=m+1}^{n-\ell+m}d_j\geq n-1-\ell+m$. On the other hand, each $d_j$ is at most $\Delta$. Thus, $\Delta(n-\ell)\geq n-1-\ell+m$, completing the proof.
\end{proof}


\subsection{Rooted forests}

Here we prove upper and lower bounds on $f_{\Delta}(n,m)$. Although the asymptotics of its weighted sum is already stated in Theorem~\ref{lemma_tree_of_bounded_degree}, bounds that we prove here will be also useful later.

\begin{claim}
There exist $0<c_1<c_2$ such that, for all positive integer $n$ and $m\leq n$,
$$
c_1^m\pr{\frac{a_{\Delta}}{e}}^{n}\binom{n-1}{m-1}n^{n-m}\leq
f_{\Delta}(n,m)\leq
c_2^m\pr{\frac{a_{\Delta}}{e}}^{n}\binom{n-1}{m-1}n^{n-m}.
$$
\label{cl:rooted_forests_bounds}
\end{claim}

\begin{proof}

Let $t^*_{\Delta}(n)\leq nt_{\Delta}(n)$ be the number of {\it rooted} trees on $[n]$ with maximum degree at most $\Delta$ such that the root has degree at most $\Delta-1$. Note that, in a tree with maximum degree at most $\Delta$, the number of vertices with degree exactly $\Delta$ is at most $2n/\Delta\leq 2n/3$. Thus, $t^*_{\Delta}(n)\geq \frac{1}{3}nt_{\Delta}(n)$. We also let $t^{**}_{\Delta}(n)=nt_{\Delta}(n)$ be the number of {\it rooted} trees on $[n]$ with maximum degree at most $\Delta$ (and without any restriction on the choice of the root). In a similar way, we let $f^{**}_{\Delta}(n,m)$ be the number of labelled rooted forests on $[n]$ with $m$ components and maximum degree at most $\Delta$ (and, again, without any restriction on the choice of the roots).

From Theorem~\ref{lemma_tree_of_bounded_degree}, we get that there are positive constants $c_1,c_2$ such that, for all $n$, 
\begin{equation}
c_1\pr{a_{\Delta}/e}^n n^{n-1}\leq t^*_{\Delta}(n)\leq c_2\pr{a_{\Delta}/e}^n n^{n-1}.
\label{eq:trees_rooted_bounds}
\end{equation} 
Since a rooted forest with $m$ components can be viewed as an unordered partition of the set of vertices into $m$ subsets with a rooted tree on each subset, $f_{\Delta}(n,m)$ can be expressed 
as follows:
\begin{align*}
f_{\Delta}(n,m)
&=\frac{n!}{m!}[x^n]\pr{\sum\limits_{s=1}^{\infty}\frac{t^*_{\Delta}(s)x^s}{s!}}^m\stackrel{(*)}\leq
\frac{n!}{m!}[x^n]\pr{\sum\limits_{s=1}^{\infty}\frac{c_2\pr{\frac{a_{\Delta}}{e}}^{s}s^{s-1}x^s}{s!}}^m\\
&=\frac{c_2^mn!}{m!}[x^n]\pr{\sum\limits_{s=1}^{\infty}\frac{s^{s-1}}{s!}\pr{\frac{a_{\Delta}}{e}x}^{s}}^m=
c_2^m\left(\frac{a_\Delta}{e}\right)^n \frac{n!}{m!}[x^n]\left(\sum_{s=1}^{\infty} \frac{s^{s-1}}{s!}x^s\right)\\
&=c_2^m\left(\frac{a_\Delta}{e}\right)^n \frac{n!}{m!}[x^n]\left(\sum_{s=1}^{\infty} \frac{t^{**}_s(s)}{s!}x^s\right)\\
&=c_2^m\pr{\frac{a_{\Delta}}{e}}^{n}f^{**}_n(n,m)\stackrel{(**)}=c_2^m\pr{\frac{a_{\Delta}}{e}}^{n}\binom{n-1}{m-1}n^{n-m},
\end{align*}
where the equality (**) holds true since $f^{**}_n(n,m)=\binom{n-1}{m-1}n^{n-m}$, see~\cite[Equation (4.2)]{Moon1970}. This completes the proof of the upper bound in Claim~\ref{cl:rooted_forests_bounds}. The proof of the lower bound is literally the same, we only need to replace the inequality (*) with the opposite inequality and $c_2$ with $c_1$ due to~\eqref{eq:trees_rooted_bounds}.
\end{proof}

\begin{remark}
It is easy to see that the same proof can be applied show that the statement of Claim~\ref{cl:rooted_forests_bounds} also holds for the number of rooted forests  on $[n]$ with $m$ components, maximum degree at most $\Delta$, and without any restriction on the choice of the roots.
\end{remark}

\subsection{Maximum value of an auxiliary function}

Let $y_1,y_2,y_3>0$. Consider the following function $\rho$ defined in $(0,\infty)$:
\begin{equation}\label{g}
\rho(x):=\rho_{y_1,y_2,y_3}(x)=\pr{\frac{y_1 y_2^{x}}{x^{y_3}}}^{x}.
\end{equation}
We will use the following technical claim.
\begin{claim}
For any $b\ge a >0$,
\begin{equation}
\label{eq:function_g_upper_bound_0}
\max\limits_{a\le x \le b}\rho(x)
\le\max\br{\exp\pr{\frac{y_3}{2}y_1^{1/y_3}},\rho(a),\rho(b)}.
\end{equation}
\label{cl:rho}
\end{claim}

\begin{proof}
If the equation $\frac{\partial}{\partial{x}}\ln{\rho(x)}=0$ has no solutions, then $\max_{a\le x\le b}\rho(x)=\max\br{\rho(a),\rho(b)},$ implying~\eqref{eq:function_g_upper_bound_0}. Now, assume that this equation has at least one solution in $[a,b]$, and let $\mathcal{C}\subset[a,b]$ be the set of solutions. Since
$$
\frac{\partial}{\partial{x}}\ln{\rho(x)}=\frac{\partial}{\partial{x}} x (-y_3\ln x + \ln y_1 + x \ln y_2)=
-y_3\ln x-y_3+\ln y_1+2x\ln y_2=\ln\pr{\frac{y_1y_2^{2x}}{(ex)^{y_3}}},
$$
for each $x\in\mathcal{C}$,
$$
\pr{\frac{y_1y_2^x}{x^{y_3}}}^2=y_1\pr{\frac{e}{x}}^{y_3} \quad\Leftrightarrow\quad
\rho(x)=\left(y_1\pr{\frac{e}{x}}^{y_3}\right)^{x/2}.
$$
Since
$$
\frac{\partial}{\partial{x}}\ln\left(y_1\pr{\frac{e}{x}}^{y_3}\right)^{x/2}=
\frac{\partial}{\partial{x}}\left(\frac{x}{2}\ln y_1+\frac{x}{2}y_3(1-\ln x)\right)=
\frac{\ln y_1}{2}-\frac{y_3\ln x}{2}=
\frac{1}{2}\ln\pr{\frac{y_1}{x^{y_3}}}=0
$$
has exactly one positive root $x^*=y_1^{1/y_3}$,
$$
\max\limits_{x\in\mathcal{C}}\rho(x)
\le\max\limits_{x>0}\left(y_1\pr{\frac{e}{x}}^{y_3}\right)^{x/2}=\left(y_1\pr{\frac{e}{x^*}}^{y_3}\right)^{x^*/2}
=\exp\pr{\frac{y_3}{2}y_1^{1/y_3}},
$$
which implies \eqref{eq:function_g_upper_bound_0}.
\end{proof}

\section{Induced subgraphs in dense random graphs}
\label{sc:const_p}

Let $p=\op{const}$, $3\leq\Delta=\op{const}$, $q=1/(1-p)$.

\subsection{Trees}
\label{sub_indiced_subtrees}

Let $Y_k$ be the number of induced trees in $G(n,p)$ on $k$ vertices and maximum degree bounded above by $\Delta$.
Then 
\begin{equation}
\mathbf{T}_{\Delta}=\max\setdef{k\in\irange{1}{n}}{Y_k>0}.
\label{eq:reduction_from_Y_to_T}
\end{equation}
For $k=O\pr{\ln{n}}$, applying the first part of Theorem~\ref{lemma_tree_of_bounded_degree}, we get
\begin{align}
\mathbb{E}Y_k
&=\binom{n}{k}(1-p)^{\binom{k}{2}-(k-1)}p^{k-1}t_{\Delta}(k)=\label{eq:expect_Y_dense}\\
&=\exp\pr{k\ln{n}-\frac{5}{2}\ln{k}+k\ln{a_{\Delta}}-\binom{k}{2}\ln{q}+k\ln{(pq)}+O(1)}=\\
&=\exp\pr{\frac{k}{2}\ln{q}\pr{2\log_{q}\pr{a_{\Delta}np}+3-k}-\frac{5}{2}\ln{k}+O(1)}.\notag
\end{align}
Fix arbitrary $\varepsilon>0$. For $k=\floor{2\log_{q}\pr{a_{\Delta}np}+3-\varepsilon}$:
$$
\mathbb{E}Y_{k}
>\exp\pr{\frac{k}{2}\varepsilon\ln{q}-\frac{5}{2}\ln{k}+O(1)}\to\infty\quad\text{as}\quad n\to\infty.
$$
Furthermore, for $k=\ceil{2\log_{q}\pr{a_{\Delta}np}+3}$:
\begin{equation}
\mathbb{E}Y_{k}\le\exp\pr{-\frac{5}{2}\ln{k}+O(1)}=o(1).
\label{eq:Markov_trees_dense}
\end{equation}
By Markov's inequality, whp there are no induced trees in $G(n,p)$ on $\ceil{2\log_{q}\pr{a_{\Delta}np}+3}$ vertices. As a consequence, there are also no larger induced trees whp.

We further assume that $k=\floor{2\log_{q}\pr{a_{\Delta}np}+3-\varepsilon}$. Let us estimate the variance $\D{Y_k}$:
$$
 \mathbb{E}Y_k(Y_k-1)-\pr{\mathbb{E}Y_k}^2\le\sum\limits_{\ell=2}^{k-1}{F_{\ell}},
$$
where $F_{\ell}$ is the expected value of the number of (ordered) pairs of trees with maximum degrees at most $\Delta$ and of
size $k$ with $\ell$ common vertices. Each such pair of trees can be obtained as follows:
\begin{itemize}
\item choose a tree $T_1$ on $k$ vertices;
\item choose an induced subgraph $F\subset{T_1}$ on $\ell$ vertices;
\item choose another tree $T_2$ on $k$ vertices containing $F$ as an induced subgraph. 
\end{itemize}
Letting 
$$
 t_{\Delta}(k,\ell,r):=\max\limits_{\text{forest }F\text{ on $[k]$: }\abs{V(F)}=\ell,\,\abs{E(F)}=r}t_{\Delta}(k,F),
$$
where $t_{\Delta}(k,F)$ is the number of trees on $[k]$ with maximum degree at most $\Delta$ and containing $F$ as an induced subforest, we get
$$
F_{\ell}\leq\binom{n}{k}t_{\Delta}(k)\binom{k}{\ell}\binom{n-k}{k-\ell}
\max\limits_{r\in\irange{0}{\ell-1}}t_{\Delta}(k,\ell,r)p^{2(k-1)-r}(1-p)^{2\binom{k}{2}-\binom{\ell}{2}-2(k-1)+r}.
$$
Therefore,
\begin{equation}
\frac{F_\ell}{\pr{\mathbb{E}Y_k}^2}\leq
\frac{\binom{k}{\ell}\binom{n-k}{k-\ell}q^{\binom{\ell}{2}}
\max\limits_{r\in\irange{0}{\ell-1}}(pq)^{-r}t_{\Delta}(k,\ell,r)}{\binom{n}{k}t_{\Delta}(k)}.
\label{eq:F_ell_trees_dense}
\end{equation}
Assume that $n$ is sufficiently large and consider the following $3$ cases:
\begin{enumerate}
    \item $2\le{\ell}\le{k}-7\log_{q}k$;
    \item $k-7\log_{q}k<\ell<k-c$;
    \item $k-c\le{\ell}\le{k-1}$,
\end{enumerate}
where $c$ is a positive large enough integer constant.

\paragraph{Case $1$.} Let $2\le{\ell}\le{k}-7\log_{q}k$.  Obviously, $t_{\Delta}(k,\ell,r)\le{t_{\Delta}(k)}$. Then, due to~\eqref{eq:F_ell_trees_dense},
\begin{align*}
\frac{F_\ell}{\pr{\mathbb{E}Y_k}^2}\le
\frac{\binom{k}{\ell}\binom{n-k}{k-\ell}q^{\binom{\ell}{2}}
(pq)^{-\ell}}{\binom{n}{k}}
&=
\frac{O(1)\left(\frac{ek}{\ell}\right)^{\ell}\left(\frac{en}{k-\ell}\right)^{k-\ell}q^{\binom{\ell}{2}}
\min\{1,pq\}^{-\ell}}{(en/k)^k}\\
&=
\frac{O(1)k^{2\ell}\left(\frac{k}{k-\ell}\right)^{k-\ell}q^{\binom{\ell}{2}}
(1+1/pq)^{\ell}}{\ell^{\ell}n^{\ell}}\leq
\pr{\frac{k^2}{n}(1+1/pq)q^{\frac{\ell}{2}}}^\ell. 
\end{align*}
Since $q^{\ell/2}\leq q^{k/2-3.5\log_qk}=O\left(nk^{-3.5}\right)$, we get
\begin{equation}
\frac{\sum\limits_{\ell=2}^{\floor{{k}-7\log_{q}k}}{F_\ell}}{\pr{\mathbb{E}Y_k}^2}\le
\sum\limits_{\ell=1}^{\infty}{\pr{\frac{1+1/pq}{k}}^\ell}=
\frac{1}{\frac{k}{1+1/pq}-1}=o(1).
\label{eq:Chebyshev_dense_trees_1}
\end{equation}

\paragraph{Case $2$.} Let $k-7\log_{q}k<\ell<k-c$. Here, we apply Lemma~\ref{lemma_labelled_trees_with_a_given_subgraph} and get
\begin{align*}
 t_{\Delta}(k,\ell,r)\leq t_k(k,\ell,r) &\leq\max_{\ell_1,\ldots,\ell_{\ell-r}\geq 2:\,\ell_1+\ldots+\ell_{\ell-r}=\ell}\ell_1\ldots\ell_{\ell-r} k^{k-\ell-1}(k-\ell)^{\ell-r-1}\\
 &\leq\pr{\frac{\ell}{\ell-r}}^{\ell-r}{k^{k-\ell-1}}(k-\ell)^{\ell-r-1}.
\end{align*}
Letting
$$
g(k,\ell,r):=(pq)^{-r}\pr{\frac{\ell}{\ell-r}}^{\ell-r}{k^{k-\ell-1}}(k-\ell)^{\ell-r-1},
$$
we get
$$
\frac{\partial}{\partial{r}}\ln(g(k,\ell,r))=\ln\pr{\frac{e}{pq}\cdot\frac{\ell-r}{\ell(k-\ell)}}.
$$
Thus $g(k,\ell,r)$ as a function of $r$ achieves its maximum on $(-\infty,\ell)$ at $r^*=\ell\cdot\pr{1-\frac{pq}{e}(k-\ell)}.$ If $\ell<k-\frac{e}{pq}$, then $r^*<0$. Since we may choose $c$ as large as we want, we assume that $c>\frac{e}{pq}$. Then
$$
\max\limits_{r\in\irange{0}{\ell-1}}g(k,\ell,r)=g(k,\ell,0)={k^{k-\ell-1}}(k-\ell)^{\ell-1}
$$
and, therefore, due to~\eqref{eq:F_ell_trees_dense},
\begin{align*}
\frac{F_\ell}{\pr{\mathbb{E} Y_k}^2}
&\leq
\frac{\binom{k}{\ell}\binom{n-k}{k-\ell}q^{\binom{\ell}{2}}
{k^{k-\ell-1}}(k-\ell)^{\ell-1}}{\binom{n}{k}t_{\Delta}(k)}
\leq\frac{\pr{\frac{ek}{\ell}}^\ell\frac{n^{k-\ell}}{(k-\ell)!}q^{\frac{\ell^2}{2}}
{k^{k-\ell-1}}(k-\ell)^{\ell-1}}{\frac{n^k}{k!}\pr{\frac{a_{\Delta}}{e}}^{k}k^{k-2}}\\
&\leq
\frac{k}{k-\ell}\pr{\frac{e}{a_{\Delta}}}^{k}
\pr{\frac{ek(k-\ell)}{n\ell}q^{\frac{\ell}{2}}}^\ell\leq
e^k\pr{\frac{ek(k-\ell)}{n\ell}q^{\frac{\ell}{2}}}^\ell.
\end{align*}
It is easy to check that
$$
\argmax\limits_{\ell\in(k-7\log_{q}k,k)}{\frac{ek(k-\ell)}{n\ell}q^{\frac{\ell}{2}}}=k-\frac{2}{\ln{q}}+o(1)
$$
and, assuming that $c>\frac{2}{\ln{q}}$, for $\ell\in\pr{k-7\log_{q}k,k-c}$,
$$
\frac{\partial}{\partial{\ell}}\ln\pr{\frac{ek(k-\ell)}{n\ell}q^{\frac{\ell}{2}}}=\frac{\ln{q}}{2}-\frac{k}{\ell(k-\ell)}>0.
$$
Thus,
$$
\frac{F_\ell}{\pr{\mathbb{E}Y_k}^2}\leq
e^k\pr{\frac{eck}{n(k-c)}q^{\frac{k-c}{2}}}^\ell\leq
\pr{\frac{e^3c}{n}q^{\frac{2\log_{q}(a_{\Delta}np)+3-c}{2}}}^\ell=
\pr{e^3c{a_{\Delta}p}q^{\frac{3-c}{2}}}^\ell.
$$
Obviously, $e^3c{a_{\Delta}p}q^{\frac{3-c}{2}}<1$ for sufficiently large (constant) $c$. Therefore,
\begin{equation}
\frac{\sum\limits_{\ell=\floor{{k}-7\log_{q}k}+1}^{k-c-1}{F_\ell}}{\pr{\mathbb{E}Y_k}^2}\leq
\sum\limits_{\ell=\floor{{k}-7\log_{q}k}+1}^{+\infty}{\pr{e^3c{a_{\Delta}p}q^{\frac{3-c}{2}}}^\ell}\leq
\frac{\pr{e^3c{a_{\Delta}p}q^{\frac{3-c}{2}}}^{{k}-7\log_{q}k}}{1-e^3c{a_{\Delta}p}q^{\frac{3-c}{2}}}=o(1).
\label{eq:Chebyshev_dense_trees_2}
\end{equation}

\paragraph{Case $3$.} Let us assume that $c$ is any integer constant and $\ell=k-c$. Due to~\eqref{eq:F_ell_trees_dense}, we get
$$
\frac{F_\ell}{\pr{\mathbb{E}Y_k}^2}\le
\frac{k^{c}n^{c}q^{\frac{k^2}{2}-\pr{c+\frac{1}{2}}k+\frac{c(c+1)}{2}}
\max\limits_{r\in\irange{0}{\ell-1}}(pq)^{-r}t_{\Delta}(k,\ell,r)}{\frac{n^k}{k!}\pr{\frac{a_{\Delta}}{e}}^{k}k^{k-2}}.
$$
Let us consider two factors in the right-hand side of the last inequality separately. First,
$$
\frac{k^{c}n^{c}q^{\frac{k^2}{2}-\pr{c+\frac{1}{2}}k+\frac{c(c+1)}{2}}}
{\frac{n^k}{k!}\pr{\frac{a_{\Delta}}{e}}^{k}k^{k-2}}=
\frac{q^{\frac{k}{2}\pr{k-2c-1}}}{n^{k-c+o(1)}a_{\Delta}^{k}}
\le
\frac{\pr{q^{\frac{3-\varepsilon}{2}}a_{\Delta}np}^{k-2c-1}}{n^{k-c+o(1)}a_{\Delta}^{k}}
=
\frac{q^{\frac{1-\varepsilon}{2}k}(pq)^k}{n^{c+1+o(1)}}
\le
\frac{(pq)^k}{n},
$$
where the last inequality follows from the fact that $q^{k(1-\varepsilon)/2}\leq(a_{\Delta} npq^{1.5})^{(1-\varepsilon)/2} \ll n^{c+o(1)}$. Since, due to Claim~\ref{cl:cl3}, $t_{\Delta}(k,\ell,r)$ is non-zero only when $\ell-r\le1+(k-\ell)(\Delta-1)\le{c}\Delta$,
$$
\max\limits_{r\in\irange{0}{\ell-1}}(pq)^{-r}t_{\Delta}(k,\ell,r)
\le\max\limits_{r\in\irange{\ell-c\Delta}{\ell-1}}(pq)^{-r}\pr{\frac{\ell}{\ell-r}}^{\ell-r}k^{k-\ell}(k-\ell)^{\ell-r-1}=(pq)^{-k}k^{O(1)}.
$$
Thus,
\begin{equation}
\frac{F_\ell}{\pr{\mathbb{E}Y_k}^2}\le
\frac{(pq)^k}{n}(pq)^{-k}k^{O(1)}=
\frac{k^{O(1)}}{n}=o(1).
\label{eq:Chebyshev_dense_trees_3}
\end{equation}

Summing up, from~\eqref{eq:Chebyshev_dense_trees_1},~\eqref{eq:Chebyshev_dense_trees_2},~and~\eqref{eq:Chebyshev_dense_trees_3}, we get
$$
\frac{\sum\limits_{\ell=2}^{k-1}{F_\ell}}{\pr{\mathbb{E}Y_k}^2}=
\frac{\sum\limits_{\ell=2}^{\floor{{k}-7\log_{q}k}}{F_\ell}}{\pr{\mathbb{E}Y_k}^2}+
\frac{\sum\limits_{\ell=\floor{{k}-7\log_{q}k}+1}^{k-c-1}{F_\ell}}{\pr{\mathbb{E}Y_k}^2}+
\frac{\sum\limits_{\ell=k-c}^{k-1}{F_\ell}}{\pr{\mathbb{E}Y_k}^2}=o(1).
$$
Thus, by Chebyshev's inequality,
$$
\mathbb{P}(Y_k=0)\leq\frac{\D{Y_k}}{\pr{\mathbb{E}Y_k}^2}\le
\frac{\mathbb{E}Y_k(Y_k-1)-\pr{\mathbb{E}Y_k}^2+\mathbb{E}Y_k}{\pr{\mathbb{E}Y_k}^2}\le
\frac{\sum\limits_{\ell=2}^{k-1}{F_\ell}}{\pr{\mathbb{E}Y_k}^2}+\frac{1}{\mathbb{E}Y_k}=o(1).
$$

Recalling~\eqref{eq:reduction_from_Y_to_T} and~\eqref{eq:Markov_trees_dense}, we conclude that that for any $\varepsilon>0$, whp
$$
\floor{2\log_{q}\pr{a_{\Delta}np}+3-\varepsilon}\le
\mathbf{T}_{\Delta}\le
\ceil{2\log_{q}\pr{a_{\Delta}np}+2},
$$
completing the proof of Theorem \ref{th:concentration_max_ind_tree_bounded_deg}.


\subsection{Forests}
\label{sc:dense_forests_proof_TH}

Let $Z_k$ be the number of induced rooted forests in $G(n,p)$ on $k$ vertices, with maximum degree at most $\Delta$, and roots having degrees at most $\Delta-1$. Then 
\begin{equation}
\mathbf{F}_{\Delta}=\max\setdef{k\in\irange{1}{n}}{Z_k>0}.
\label{eq:reduction_from_Z_to_F}
\end{equation}
By Claim~\ref{cl:rooted_forests_bounds},
\begin{align}
\mathbb{E}Z_{k}
&\le\binom{n}{k}\sum\limits_{m=1}^{k}
\binom{k-1}{m-1}k^{k-m}c_2^m\pr{\frac{a_{\Delta}}{e}}^{k} p^{k-m}q^{-\binom{k}{2}+(k-m)}\notag\\
&=\pr{\frac{a_{\Delta}}{e}}^{k}\binom{n}{k}q^{-\binom{k}{2}}
\sum\limits_{m=1}^{k}\binom{k-1}{m-1}(kpq)^{k-m}c_2^m\notag\\
&=c_2\pr{\frac{a_{\Delta}}{e}}^{k}\binom{n}{k}q^{-\binom{k}{2}}(kpq+c_2)^{k-1}\label{eq:Z_upper_c2}\\
&=\Theta\left(\pr{\frac{a_{\Delta}}{e}}^{k}\binom{n}{k}q^{-\binom{k}{2}}(kpq)^{k-1}\right)\stackrel{\eqref{eq:expect_Y_dense}}=\Theta(k\cdot\mathbb{E}Y_k),\notag
\end{align}
where $Y_k$ is the number of induced trees in $G(n,p)$ on $k$ vertices and maximum degree bounded above by $\Delta$.
Due to~\eqref{eq:Markov_trees_dense}, we get that, for $k=\ceil{2\log_{q}\pr{a_{\Delta}np}+3}$,
$$
\mathbb{E}Z_{k}\le\exp\pr{-\frac{3}{2}\ln{k}+O(1)}=o(1).
$$
By Markov's inequality whp 
 $\mathbf{F}_{\Delta}\leq k-1$. Since tree is a forest, $\mathbf{T}_{\Delta}\leq\mathbf{F}_{\Delta}$. Thus, the lower bound in the first part of Theorem~\ref{lemma_tree_of_bounded_degree} implies the lower bound in the second part of Theorem~\ref{lemma_tree_of_bounded_degree} as well, completing the proof.

\section{Induced forests in sparse random graphs}
\label{sc:small_p}

Let $p=p(n)\in(0,1)$ be arbitrary, $3\leq\Delta=\mathrm{const}$, $q:=1/(1-p)$.

As in Section~\ref{sc:dense_forests_proof_TH}, we denote by $Z_k$ the number of induced rooted subforests in $G(n,p)$ on $k$ vertices with maximum degree at most $\Delta$ and with roots having degrees at most $\Delta-1$. Then~\eqref{eq:reduction_from_Z_to_F} holds. Fix a small $\varepsilon>0$ and set
\begin{align*}
k_{+\varepsilon}&:=\ceil{2\log_q(a_{\Delta}np(1+\varepsilon))+3},\\
k_{-\varepsilon}&:=\floor{2\log_q(a_{\Delta}np(1-\varepsilon))+3}.
\end{align*}

The rest of the proof is organised as follows. In Section~\ref{sc:small_p_proof_exp}, we estimate $\mathbb{E}Z_k$ and prove, in particular, that it is small when $k=k_{+\varepsilon}$ and large when $k=k_{-\varepsilon}$. In Section~\ref{sc:small_p_proof_var}, we bound $\mathrm{Var}Z_k$. Finally, in Section~\ref{sc:small_p_proof_Talagrand}, we combine the results from the previous two sections and apply Talagrand's inequality to finish the proof of Theorem~\ref{th:concentration_max_ind_rooted_forests_bounded_degree}.

\subsection{Expectation}
\label{sc:small_p_proof_exp}

By Claim~\ref{cl:rooted_forests_bounds}, in the same way as in~\eqref{eq:Z_upper_c2}, we get that, for all $k\in[n]$,
\begin{align*}
\label{eq:e_z_t_lower_upper_bounds}
c_1\pr{\frac{a_{\Delta}}{e}}^{k}\binom{n}{k}q^{-\binom{k}{2}}(kpq+c_1)^{k-1}
\le\mathbb{E}Z_k
&\le c_2\pr{\frac{a_{\Delta}}{e}}^{k}\binom{n}{k}q^{-\binom{k}{2}}(kpq+c_2)^{k-1}
\end{align*}
or, equivalently,
\begin{equation}
\label{eq:e_z_t}
\mathbb{E}Z_k=\pr{\frac{a_{\Delta}}{e}}^{k}\binom{n}{k}q^{-\binom{k}{2}}\left(kpq+e^{O(1)}\right)^{k-1}.
\end{equation}
Applying Stirling's formula, for $k\leq n/2$, we obtain
$$
\binom{n}{k}=\pr{\frac{en}{k}\exp\pr{O\pr{\frac{k/n}{1-k/n}}+O\pr{\frac{\ln{k}}{k}}}}^{k}
$$
and
\begin{align*}
\mathbb{E}Z_k &=\pr{\frac{a_{\Delta} n}{k}\exp\pr{O\pr{\frac{k/n}{1-k/n}}+O\pr{\frac{\ln{k}}{k}}}q^{-\frac{k-1}{2}}(kpq+e^{O(1)})^{1-1/k}}^{k}\\
&=\pr{a_{\Delta}npq^{-\frac{k-3}{2}}\exp\pr{O\pr{\frac{k/n}{1-k/n}}+O\pr{\frac{\ln{k}}{k}}+O\pr{\frac{1}{kp}}}}^{k}.
\end{align*}

\begin{claim}
For any fixed $\varepsilon>0$ there exists a sufficiently large constant $C_{\varepsilon}>0$ such that if $C_{\varepsilon}/n<p<1-\varepsilon$, then
$\mathbb{E}Z_{k_{+\varepsilon}}\to 0$ and $\mathbb{E}Z_{k_{-\varepsilon}}\to\infty$ as $n\to\infty$.
\label{claim:expectation_n_ind_forests_bounded_degree_at_k_pm_eps}
\end{claim}

\begin{proof}
It is easy to see that if $C_{\varepsilon}/n<p<1-\varepsilon$ and $C_{\varepsilon}$ is sufficiently large, then
$$
\frac{k_{\pm\varepsilon}}{n}\leq
\frac{2\ln(a_{\Delta}np(1+\varepsilon))}{n\ln q}+\frac{4}{n}\leq
\frac{2\ln(enp)+1}{np}+\frac{4}{np}\leq
\frac{2\ln(np)+7}{np}\le
\frac{3\ln{C_{\varepsilon}}}{C_{\varepsilon}}
$$
and, since $p/\ln q\geq 1-p$,
\begin{equation}
k_{\pm\varepsilon}p\ge2(1-p)\ln(np(1-\varepsilon))
>2\varepsilon\ln(C_{\varepsilon}(1-\varepsilon))
>\varepsilon\ln{C_{\varepsilon}}.
\label{eq:kp_large}
\end{equation}
Since $k_{\pm\varepsilon}\to\infty$ as $n\to\infty$ and $C_{\varepsilon}$ is sufficiently large,
$$
\abs{\exp\pr{O\pr{\frac{k_{\pm\varepsilon}/n}{1-k_{\pm\varepsilon}/n}}+O\pr{\frac{\ln{k_{\pm\varepsilon}}}{k_{\pm\varepsilon}}}+O\pr{\frac{1}{k_{\pm\varepsilon}p}}}-1}\le\frac{\varepsilon}{2}.
$$
Thus,
\begin{align}
\mathbb{E}Z_{k_{+\varepsilon}}&\le
\pr{a_{\Delta}npq^{-\frac{k_{+\varepsilon}-3}{2}}(1+\varepsilon/2)}^{k_{+\varepsilon}}\le
\pr{\frac{1+\varepsilon/2}{1+\varepsilon}}^{k_{+\varepsilon}}
\to0\quad\text{and}\notag\\
\mathbb{E}Z_{k_{-\varepsilon}}&\ge
\pr{a_{\Delta}npq^{-\frac{k_{-\varepsilon}-3}{2}}(1-\varepsilon/2)}^{k_{-\varepsilon}}\ge
\pr{\frac{1-\varepsilon/2}{1-\varepsilon}}^{k_{-\varepsilon}}
\to\infty
\label{eq:forests_sparse_exp_large}
\end{align}
as $n\to\infty$, completing the proof.
\end{proof}

\subsection{Variance}
\label{sc:small_p_proof_var}

This section is devoted to the proof of the following lemma.

\begin{lemma}
For any fixed $\varepsilon>0$ there exists a sufficiently large constant $C_{\varepsilon}>0$ such that if $C_{\varepsilon}/n<p<1-\varepsilon$, then
$$
\frac{\mathrm{Var} Z_{k_{-\varepsilon}}}{(\mathbb{E}Z_{k_{-\varepsilon}})^2}
\le\exp\pr{\frac{O(k_{-\varepsilon})}{(\ln(np))^5}}.
$$
If, additionally, $p> n^{-1/2+\varepsilon}$, then
$
\mathrm{Var} Z_{k_{-\varepsilon}}/(\mathbb{E} Z_{k_{-\varepsilon}})^2
=o(1).
$
\label{lm:variance_sparse_forests}
\end{lemma}

We need the following two claims.

\begin{claim}
For any fixed $\varepsilon>0$, there exists a sufficiently large constant $C_{\varepsilon}>0$ such that if $C_{\varepsilon}/n<p<1-\varepsilon$, then, for all $\ell\le\sqrt{k_{-\varepsilon}/p}$,
$$
\frac{\mathbb{E} Z_{k_{-\varepsilon}-\ell}}{\mathbb{E} Z_{k_{-\varepsilon}}}q^{-\ell(k_{-\varepsilon}-\ell)}\leq 2e^{3\ell+o(1)}\pr{\frac{q^{\frac{\ell-3}{2}}}{np}}^{\ell}\pr{1-\frac{k_{-\varepsilon}}{n}}^{-k_{-\varepsilon}}.
$$
\label{cl:forests_sparse_variance_cl1}
\end{claim}

\begin{proof}
By Theorem~\ref{lemma_tree_of_bounded_degree}, for every $k\geq\frac{1}{pq}$ (note that $k=k_{-}$ and $k=k_- -\ell$ are such when $C_{\varepsilon}$ is large enough due to~\eqref{eq:kp_large}), letting $w=1/(pq)$, we get
\begin{align*}
\mathbb{E} Z_k={n\choose k}\sum_{m=1}^k f_{\Delta}(k,m)p^{k-m}(1-p)^{{k\choose 2}-k+m}
&=\binom{n}{k}q^{-\binom{k}{2}}(p q)^{k}\sum_{m=1}^l  f_{\Delta}(k,m)(1/pq)^m\\
&=\binom{n}{k}q^{-\binom{k}{2}}(p q)^{k-1}(k-1)!A(r_t),
\end{align*}
where 
$$
 A(r_k)=\frac{\left(\gamma_{\Delta}(r_k)\right)^t e^{\frac{r_k}{p q}+o(1)}}{r^{k-1}_k \sqrt{2 \pi \beta(r_k) k}}
$$ 
and $r_k$ (dependence on $n$ is implicit) is the unique positive solution of $\frac{r \gamma_{\Delta}^{\prime}(r)}{\gamma_{\Delta}(r)}+\frac{r}{k p q}=1$.

Let us set $k=k_-$ and estimate $\frac{\mathbb{E}Z_{k-\ell}}{\mathbb{E}Z_k}$:
\begin{align*}
\frac{\mathbb{E}Z_{k-\ell}}{\mathbb{E}Z_k}\cdot q^{-\ell(k-\ell)}&\leq \frac{\frac{n^{k-\ell}}{(k-\ell)!}}{\frac{(n-k)^k}{k!}}q^{-\binom{k-\ell}{2}+\binom{k}{2}-\ell(k-\ell)}\pr{pq}^{-\ell}\frac{(k-\ell-1)!}{(k-1)!}\frac{A(r_{k-\ell})}{A(r_k)}\\
&=
\frac{k}{k-\ell}\frac{\pr{np}^{-\ell}}{\pr{1-\frac{k}{n}}^k}q^{\frac{\ell^2}{2}-\frac{3\ell}{2}}\frac{A(r_{k-\ell})}{A(r_k)}\stackrel{\eqref{eq:kp_large}}\leq
2\frac{A(r_{k-\ell})}{A(r_k)}\pr{\frac{q^{\frac{\ell-3}{2}}}{np}}^{\ell}\pr{1-\frac{k}{n}}^{-k}.
\end{align*}

In order to estimate $\frac{A(r_{k-\ell})}{A(r_k)}$, let us view $r_t$ is a function of $t$ which is defined implicitly on $(0,\infty)$ as the positive solution of the equation (the fact that such a solution exists and unique is explained in Section~\ref{sc:part2_proof})
$$
 \varphi(t,r):=\frac{r \gamma_{\Delta}^{\prime}(r)}{\gamma_{\Delta}(r)}+\frac{r}{t p q}-1=0.
$$
Let us observe the following properties of the function $r_t$.

\begin{prop}
$r_t>1$ when $|t-k|\le\sqrt{k/p}$.
\label{prop:prop1}
\end{prop}
\begin{proof}
We first note that $tpq>2^{\Delta}(\Delta-1)!e$ for sufficiently large $C_{\varepsilon}$ and $n$, due to (\ref{eq:kp_large}).

Now, assume that $1/2<r_t\leq 1$. Then 
$$
\frac{r_t \gamma_{\Delta}^{\prime}(r_t)}{\gamma_{\Delta}(r_t)}\leq \frac{\gamma_{\Delta}^{\prime}(r_t)}{\gamma_{\Delta}(r_t)}=\frac{\gamma_{\Delta}(r_t)-\frac{r_t^{\Delta-1}}{(\Delta-1)!}}{\gamma_{\Delta}(r_t)}\leq 1-\frac{r_t^{\Delta-1}}{(\Delta-1)!e^{r_t}}\leq 1-\frac{1}{2^{\Delta-1}(\Delta-1)!e}.
$$ 
Hence, 
$$
\frac{r_t \gamma_{\Delta}^{\prime}(r_t)}{\gamma_{\Delta}(r_t)}+\frac{r_t}{t p q}<1-\frac{1}{2^{\Delta}(\Delta-1)!e}<1.
$$ 
Also, $r_t$ cannot be less than $1/2$, because in that case  
$$
\frac{r_t \gamma_{\Delta}^{\prime}(r_t)}{\gamma_{\Delta}(r_t)}+\frac{r_t}{t p q}<\frac{\gamma'_{\Delta}(r_t)}{2\gamma_{\Delta}(r_t)}+\frac{1}{2t p q}<\frac{1}{2}+\frac{1}{2^{\Delta+1}(\Delta-1)!e}<1.
$$
\end{proof}

\begin{prop}
$0<r'_t\leq\frac{r_t e^{2r_t}}{t^2pq}$.
\label{prop:prop2}
\end{prop}

\begin{proof}
By the implicit function theorem, 
\begin{equation}\label{Impl_thm}
    r^{\prime}_t=-\frac{\varphi_t^{\prime}(t, r_t)}{\varphi_r^{\prime}(t, r_t)}.
\end{equation}
Let us show that the denominator in \eqref{Impl_thm} is not equal to zero. Clearly,
\begin{align*}
\varphi_r^{\prime}(t, r_t) &=\frac{\gamma_{\Delta}'(r_t)\gamma_{\Delta}(r_t)+r_t \gamma_{\Delta}''(r_t)\gamma_{\Delta}(r_t)-r_t(\gamma_{\Delta}'(r_t))^2}{(\gamma_{\Delta}(r_t))^2}+\frac{1}{tpq}\\
&=\frac{\gamma_{\Delta-1}(r_t)\gamma_{\Delta}(r_t)+r_t \gamma_{\Delta-2}(r_t)\gamma_{\Delta}(r_t)-r_t(\gamma_{\Delta-1}(r_t))^2}{(\gamma_{\Delta}(r_t))^2}+\frac{1}{tpq}=:\frac{\psi(r_t)}{(\gamma_{\Delta}(r_t))^2}+\frac{1}{tpq}.
\end{align*}
Due to Claim~\ref{cl:cl1}, $r_t\leq 2$. If $\Delta\geq 4$, then
$$
 \gamma_{\Delta-2}(r_t)\gamma_{\Delta}(r_t)-(\gamma_{\Delta-1}(r_t))^2=
 \gamma_{\Delta-2}(r_t)\left(\frac{r_t^{\Delta-1}}{(\Delta-1)!}-\frac{r_t^{\Delta-2}}{(\Delta-2)!}\right)-\left(\frac{r_t^{\Delta-2}}{(\Delta-2)!}\right)^2<0.
$$
Thus, 
\begin{align*}
 \psi(r_t) &\geq  
 \gamma_{\Delta-1}(r_t)(\gamma_{\Delta}(r_t)-\gamma_{\Delta-1}(r_t))+2\gamma_{\Delta-2}(r_t)\gamma_{\Delta}(r_t)- (\gamma_{\Delta-1}(r_t))^2\\
 &\geq  
 2\gamma_{\Delta-2}(r_t)\gamma_{\Delta}(r_t)- (\gamma_{\Delta-1}(r_t))^2\\
 &=(\gamma_{\Delta-2}(r_t))^2 +2\gamma_{\Delta-2}(r_t)\frac{r_t^{\Delta-1}}{(\Delta-1)!}-\left(\frac{r_t^{\Delta-2}}{(\Delta-2)!}\right)^2\geq  (\gamma_{\Delta-2}(r_t))^2\geq 1.
\end{align*}
If $\Delta=3$, then
$$
\psi(r_t)=(1+r_t)\left(1+r_t+\frac{r_t^2}{2}\right)+r_t\left(1+r_t+\frac{r_t^2}{2}\right)-r_t(1+r_t)^2=
1+2r_t+\frac{r_t^2}{2}\geq 1.
$$ 
Hence, indeed, the denominator in~\eqref{Impl_thm} does not vanish. Moreover, since $\gamma_{\Delta}(r_t)<e^{r_t}$,
$$
\varphi_r^{\prime}(t, r_t)>\frac{1}{(\gamma_{\Delta}(r_t))^2}>e^{-2r_t}.
$$
This completes the proof of Proposition~\ref{prop:prop2} since $\varphi_t^{\prime}(t, r_t)=-r_t/(t^2pq).$
\end{proof}

Now, let us derive from Proposition~\ref{prop:prop1} and Proposition~\ref{prop:prop2} the following three properties:

\begin{enumerate}

\item $r(k-\ell)\le r(k)$ due to Proposition~\ref{prop:prop2}. So,  $(\gamma_{\Delta}(r_{k-\ell}))^{k-\ell}\leq\pr{\gamma_{\Delta}(r_k})^k.$

\item By the Lagrange's mean value theorem and Claim~\ref{cl:cl1},
$$
|r_{k-\ell}-r_k|\leq \frac{110\ell}{k^2p}.
$$

\item Due to Proposition~\ref{prop:prop1},  for sufficiently large $C_{\varepsilon}$, 
$$
\pr{\frac{r_k}{r_{k-\ell}}}^k=e^{k(\ln(r_k)-\ln (r_{k-\ell}))}\leq e^{k(r_k-r_{k-\ell})}\leq e^{\frac{110\ell k}{k^2p}}=e^{\frac{110\ell}{kp}}\leq e^{\ell}.
$$
\end{enumerate}

Combining 1--3, we get that
\begin{align*}
\frac{A(r_{k-\ell})}{A(r_k)} &=2\frac{\left(\gamma_{\Delta}(r_{k-\ell})\right)^{k-\ell} e^{\frac{r_{k-\ell}}{p q}+o(1)}/\left(r_{k-\ell}^{k-\ell-1} \sqrt{2 \pi \beta(r_{k-\ell}) (k-\ell)}\right)}{\left(\gamma_{\Delta}(r_k)\right)^k e^{\frac{r_k}{p q}+o(1)}/\left(r_k^{k-1} \sqrt{2 \pi \beta(r_k) k}\right)}\\
&\leq\frac{(\gamma_{\Delta}(r_{k-\ell}))^{k-\ell}}{\pr{\gamma_{\Delta}(r_k})^k}\pr{\frac{r_k}{r_{k-\ell}}}^{k-1}\pr{r_{k-\ell}}^{\ell}e^{\frac{r_{k-\ell}-r_k}{pq}+o(1)}\sqrt{\frac{k\beta(r_k)}{(k-\ell)\beta(r_{k-\ell})}}\\
&\stackrel{\text{Claim}~\ref{cl:cl2}}\leq e^{\ell}2^{\ell}e^{o(1)}\leq e^{3\ell+o(1)}.
\end{align*}

\end{proof}

\begin{claim}
For any $\varepsilon>0$, there exists a sufficiently large constant $C_{\varepsilon}>0$ such that if $C_{\varepsilon}/n<p<1-\varepsilon$, then, for all $1\le\ell\le k_{-\varepsilon}$,
$$
\frac{\pr{\mathbb{E} Z_{k_{-\varepsilon}-\ell}}^2}{\pr{\mathbb{E} Z_{k_{-\varepsilon}} }^2}
\le(a_{\Delta}(n-k)p)^{-2\ell}q^{2k_{-\varepsilon}\ell-\ell^2-3\ell}e^{O(1/p)}k_{-\varepsilon}^3,
$$
assuming $\mathbb{E}Z_0=1$.
\label{cl:forests_sparse_variance_cl2}
\end{claim}

\begin{proof}
Set $k=k_{-\varepsilon}$. Due to \eqref{eq:e_z_t}, we get, for $\ell<k$,
\begin{align*}
\frac{\pr{\mathbb{E} Z_{k-\ell}}^2}{\pr{\mathbb{E} Z_k }^2}
&=\frac{\pr{\frac{a_{\Delta}}{e}}^{2(k-\ell)}\binom{n}{k-\ell}^{2}q^{-2\binom{k-\ell}{2}}((k-\ell)pq+\Theta(1))^{2(k-\ell-1)}}{\pr{\frac{a_{\Delta}}{e}}^{2k}\binom{n}{k}^{2}q^{-2\binom{k}{2}}(kpq+\Theta(1))^{2(k-1)}}\\
&\le\sqrt{\frac{k}{k-\ell}}\pr{\frac{a_{\Delta}}{e}}^{-2\ell}\frac{k^{2k}(n-k)^{2(n-k)}}{(n-k+\ell)^{2(n-k+\ell)}(k-\ell)^{2(k-\ell)}}\times\\
&\quad\quad\quad\quad\quad\quad\quad\quad\quad\quad\quad\times q^{-2\binom{k-\ell}{2}+2\binom{k}{2}}e^{\Theta(1/(pq))}\frac{((k-\ell)pq)^{2(k-\ell-1)}}{(kpq)^{2(k-1)}}\\
&\le(a_{\Delta}(n-k)p)^{-2\ell}q^{2k\ell-\ell^2-3\ell}e^{O(1/p)}\frac{k^3}{(k-\ell)^3}.
\end{align*}
As for the case $\ell=k$, notice that $\mathbb{E}Z_1=n$, thus
\begin{align*}
\frac{1}{\pr{\mathbb{E} Z_k }^2}
=\frac{1}{n^2}\cdot\frac{(\mathbb{E}Z_1)^2}{\pr{\mathbb{E} Z_k }^2}
&\le\frac{1}{n^2}(a_{\Delta}(n-k)p)^{-2(k-1)}q^{2k(k-1)-(k-1)^2-3(k-1)}e^{O(1/p)}k^3\\
&=(a_{\Delta}(n-k)p)^{-2k}q^{2k^2-k^2-3k}e^{O(1/p)}k^3,
\end{align*}
completing the proof.
\end{proof}

\begin{proof}[Proof of Lemma~\ref{lm:variance_sparse_forests}]

Set $k:=k_{-\varepsilon}$ and let $n$ be large enough. As usual, we write 
$$
\mathrm{Var}Z_k=\mathbb{E}Z_k(Z_k-1)+\mathbb{E}Z_k-(\mathbb{E}Z_k)^2\leq
\sum_{\ell=2}^kF_{\ell},
$$ 
where $F_{\ell}$ is the expected value of the number of (ordered) pairs $(F_1,F_2)$ of induced rooted forests of size $k$ with maximum degree at most $\Delta$ with $\ell$ common vertices. Below, we bound $F_{\ell}$ for small and large $\ell$ separately.

\paragraph{Small $\ell$.} Let $\ell\le\sqrt{k/p}$. We estimate $F_{\ell}$ in the following way:

\begin{enumerate}

\item Choose two disjoint rooted forests $F^{(1)}$ and $F^{(2)}$ of sizes $k$ and $k-\ell$ respectively with maximum degree at most $\Delta$. This contributes the factor of $\mathbb{E} Z_{k} \mathbb{E} Z_{k-\ell}$.

\item Choose a subset $L\subset V(F^{(1)})$ of size $\ell$ and fix roots in each component of $F^{(1)}[L]$ --- it contributes the factor of $\binom{k}{\ell}2^{\ell}$.

\item Construct a forest $F^{(2+)}$ on $L\sqcup V(F^{(2)})$ by drawing, for each vertex $v\in L$, at most $\Delta$ edges from $v$ to $V(F^{(2)})$ --- it contributes the factor of
$$
\left(\prod_{v\in L}\left(\sum_{0\leq s\leq\Delta}\binom{k}{s}(pq)^s\right)\right)q^{-\ell(k-\ell)}\leq
q^{-\ell(k-\ell)}\pr{(1+o(1))(kpq)^{\Delta}/\Delta!}^\ell.
$$

\item Let $C(v,G)$ denote the connected component of the graph $G$ containing its vertex $v$.

We eventually consider the following algorithm that outputs the set of roots of $F^{(2+)}$. We first order arbitrarily all vertices of $L$. Then, we follow vertices one by one, in the predescribed order, and for each vertex $v\in L$ do the following:

\begin{itemize}
\item if the tree $C(v,F^{(1)}[L])$ has already been visited, just skip this step and proceed with the next vertex;
\item otherwise, choose one of the two options: either the root of $C(v,F^{(1)}[L])$ becomes the root of $C(v,F^{(2+)})$, or the root is outside of $L$;
\item if the second choice has been made --- the root of $C(v,F^{(2+)})$ has to be outside of $L$ --- either skip the current vertex $v$ and proceed with the next vertex, or choose a neighbour $w\in V(F^{(2)})$ of $v$ and let the root of $C(w,F^{(2)})$ be the roof of $C(v,F^{(2+)})$.
\end{itemize}

Finally, components of $F^{(2)}$ that have no neighbours in $L$ and, thus, fall into $F^{(2+)}$ unchanged, inherit the roots that they have had in $F^{(2)}$. This procedure contributes the factor of $(\Delta+2)^{\ell}$.

\end{enumerate}

Hence, for sufficiently large n,
\begin{gather*}
\frac{F_{\ell}}{\pr{\mathbb{E} Z_k}^2}\leq\frac{\mathbb{E} Z_{k-\ell}}{\mathbb{E}Z_k}\binom{k}{\ell}\pr{\frac{3(\Delta+2)(kpq)^{\Delta}}{\Delta!}}^{\ell}q^{-\ell(k-\ell)}.
\end{gather*}
Claim~\ref{cl:forests_sparse_variance_cl1} implies
$$
\frac{F_{\ell}}{\pr{\mathbb{E} Z_k}^2}\leq
2e^{3\ell+o(1)}\pr{1-\frac{k}{n}}^{-k}\pr{\frac{q^{\frac{\ell-3}{2}}}{np}}^{\ell}\binom{k}{\ell}\pr{\frac{3(\Delta+2)(kpq)^{\Delta}}{\Delta!}}^{\ell}.
$$
If $p>n^{-1/2+\varepsilon}$, then $\pr{1-\frac{k}{n}}^k=1-o(1)$. Moreover, due to~\eqref{eq:kp_large},
$$
q^{\frac{\ell-3}{2}} = q^{\frac{k-3}{2} \cdot \frac{\ell-3}{k-3}}\leq
\left(q^{\frac{k-3}{2}}\right)^{\ell/k}\leq \left(q^{\frac{k-3}{2}}\right)^{1/\sqrt{kp}}\leq (enp)^{1/\sqrt{kp}} =\Theta\left((np)^{1/\sqrt{kp}}\right).
$$
Therefore,
\begin{equation}
\frac{F_{\ell}}{\pr{\mathbb{E}Z_k}^2}\leq
\pr{\Theta(1) (kpq)^{\Delta}\cdot\frac{kq^{\frac{\ell-3}{2}}}{\ell np}}^{\ell}\leq \pr{\frac{\Theta(1)(kp)^{\Delta+1}(np)^{\frac{1}{\sqrt{kp}}}}{np^2}}^{\ell}
\le\pr{\frac{\Theta(1)(\ln (np))^{\Delta+1}}{p(np)^{1-\frac{\varepsilon}{2}}}}^{\ell}<n^{-\varepsilon\ell}.
\label{eq:variance_forests_small_1}
\end{equation}
If $p\le n^{-1/2+\varepsilon}$, then $\pr{1-\frac{k}{n}}^{-k}=\left(\frac{n}{n-k}\right)^k=\left(1+\frac{k}{n-k}\right)^k\leq\exp\left(\frac{k^2}{n-k}\right)$. Therefore,
\begin{equation}
\frac{F_{\ell}}{\pr{\mathbb{E}Z_k}^2}\le
e^{\frac{k^2}{n-k}}\pr{\Theta(1)(kp)^{\Delta}\cdot\frac{kq^{\frac{\ell}{2}}}{\ell np}}^{\ell}=e^{\frac{O(k^2)}{n}}\rho_{y_1,y_2,y_3}(\ell),
\label{eq:F/E_bound}
\end{equation}
where $y_1=\Theta\left(\frac{(kp)^{\Delta}k}{np}\right)$, $y_2=q^{1/2}$, $y_3=1$, and $\rho(\ell):=\rho_{y_1,y_2,y_3}(\ell)$ is defined in~\eqref{g}. 
  Then, by Claim~\ref{cl:rho},
\begin{align*}
 \rho(\ell) & \leq
 \max\br{\exp\pr{\frac{\alpha}{2}y^{\frac{1}{\alpha}}},\rho(1),\rho(\sqrt{k/p})}\\
 &\leq
 \max\br{\exp\pr{\frac{O(1)k(kp)^{\Delta}}{np}},\frac{O(1)k(kp)^{\Delta}}{np},\exp\pr{\frac{k}{\sqrt{kp}}\ln\pr{\frac{(kp)^{\Delta+1/2}q^{k/(2\sqrt{kp})}O(1)}{np}}}}.
\end{align*}
Since $kp=O(\ln (np))$, we get that the third term is less than 1 as
$$
 \frac{(kp)^{\Delta+1/2}q^{\sqrt{k/p}/2}}{np}=
  \frac{(kp)^{\Delta+1/2}e^{\sqrt{kp}/2+o(1)}}{np}\leq
  \frac{(kp)^{\Delta+1/2}e^{\sqrt{\ln(enp)}+o(1)}}{np}
$$
can be made as small as needed due to the choice of $C_{\varepsilon}$. Also, $k(kp)^{\Delta}/(np)=k(\ln (np))^{O(1)})/(np)$ and 
$k^2/n=O(\ln(np)/(np))k$. Then, we can conclude from~\eqref{eq:F/E_bound} that, for large enough $C_{\varepsilon}$,
\begin{equation}
\frac{F_{\ell}}{\pr{\mathbb{E}Z_k}^2}\leq \exp\pr{\frac{k}{5(\ln(np))^5}}.
\label{eq:variance_forests_small_2}
\end{equation}

\paragraph{Large $\ell$.} Let $\sqrt{k/p}<\ell\le k$. Here, we estimate $F_{\ell}$ in a slightly different way:

\begin{enumerate}

\item Choose two disjoint rooted forests $F^{(1)}$ and $F^{(2)}$ of sizes $k-\ell$ with maximum degree at most $\Delta$. This contributes the factor of $(\mathbb{E} Z_{k-\ell})^2$, assuming $\mathbb{E}Z_0=1$.

\item Choose a set $L$ of size $\ell$, choose an integer $m\in[\ell]$, positive integers $\ell_1,\ldots,\ell_m$ that sum up to $\ell$, and construct on the chosen set $L$ a forest of $m$ (unrooted) trees of sizes $\ell_1,\ldots,\ell_m$ with maximum degrees at most $\Delta$. Due to Theorem~\ref{lemma_tree_of_bounded_degree}, it contributes a factor of
 $$
 \binom{n}{\ell}\sum\limits_{m=1}^{\ell}\frac{1}{m!}\left(
\sum\limits_{\ell_1,\ldots,\ell_m=1}^{\infty}\binom{\ell}{\ell_1,\ldots,\ell_m}(1+o(1))^m\prod_{j=1}^m
\pr{\frac{a_{\Delta}}{e}\ell_j}^{\ell_j-2}\right)
p^{\ell-m} q^{-\binom{\ell}{2}+m}.
$$
Denote the resulting (unrooted) forest by $F^{\cap}$.

\item  For each $j\in [m]$, choose in the $j$-th tree component of the forest $F^{\cap}$ an ordered pair of roots it at most $(\ell_1\ldots\ell_m)^2$ ways.

\item For $j\in\{1,2\}$, draw edges between $L$ and $V(F^{(j)})$ so that (1) if $\ell\leq k-\ell$, then every vertex of $L$ touches at most $\Delta$ edges going to $F^{(j)}$; (2) if $\ell> k-\ell$, then every vertex of $F^{(j)}$ touches at most $\Delta$ edges going to $L$. Then, for $j\in\{1,2\}$, this step contributes a factor of 
$$
\left(\sum_{s\leq\Delta}\binom{k}{s}(pq)^s\right)^{\min\{\ell,k-\ell\}}q^{-\ell(k-\ell)}
\leq q^{-\ell(k-\ell)}\pr{(2\ln(enp))^{\Delta}\cdot\frac{1+\varepsilon}{\Delta!}}^{\min\{\ell,k-\ell\}}
$$
since $\sum_{s\leq\Delta}\binom{k}{s}(pq)^s\leq\frac{1+\varepsilon}{\Delta!}\left(kpq\right)^{\Delta}\leq
\frac{1+\varepsilon}{\Delta!}\left(\frac{2\ln(a_{\Delta}np)}{\ln q}pq\right)^{\Delta}\leq\frac{1+\varepsilon}{\Delta!}\left(2\ln(enp)\right)^{\Delta}$.

\item Clearly, for each $j\in\{1,2\}$, we have already overestimated the expected number of pairs of forests that have $\ell$ common vertices. It remains to choose roots. We do it in the same way as in the case of small $\ell$, this gives the factor of $(\Delta+2)^{2\min\{\ell,k-\ell\}}$.
\end{enumerate}

Multiplying the factors that we have just obtained, we get
\begin{align*}
\frac{F_{\ell}}{(\mathbb{E} Z_k)^2} &=
\frac{\pr{\mathbb{E} Z_{k-\ell}}^2}{\pr{\mathbb{E}Z_k }^2}\binom{n}{\ell}\sum\limits_{m=1}^{\ell}\frac{(1+o(1))^m}{m!}
\sum\limits_{\ell_1,\ldots,\ell_m=1}^{\infty}\binom{\ell}{\ell_1,\ldots,\ell_m}
\ell_1^{\ell_1}\ldots{\ell_m^{\ell_m}}\pr{\frac{pa_{\Delta}}{e}}^{\ell}\times\\
&\quad\quad\quad\times\left(2(\ln(enp))^{\Delta}\cdot\frac{1+\varepsilon}{\Delta!}\right)^{2\min\{\ell,k-\ell\}}\pr{\Delta+2}^{2\min\{\ell,k-\ell\}}q^{-\ell^2/2-2\ell(k-\ell)+\ell/2}\pr{\frac{p a^2_{\Delta}}{q e^2}}^{-m}
\end{align*}
Since 
$$
 \sum\limits_{\ell_1,\ldots,\ell_m=1}^{\infty}\binom{\ell}{\ell_1,\ldots,\ell_m}
\ell_1^{\ell_1-1}\ldots{\ell_m^{\ell_m-1}}=m!{\ell-1\choose m-1}\ell^{\ell-m}
$$
(see~\cite[Equation (3.4)]{Moon1970}), we get 
$$
 \sum\limits_{\ell_1,\ldots,\ell_m=1}^{\infty}\binom{\ell}{\ell_1,\ldots,\ell_m}
\ell_1^{\ell_1}\ldots{\ell_m^{\ell_m}}\leq m!{\ell-1\choose m-1}\ell^{\ell-m}\left(\frac{\ell}{m}\right)^m
\leq m!{\ell\choose m}\frac{\ell^{\ell}}{m^m}\leq\frac{\ell^{m+\ell}}{m^m}.
$$
Thus,
\begin{align*}
\frac{F_{\ell}}{(\mathbb{E} Z_k)^2} &\stackrel{\text{Claim~\ref{cl:forests_sparse_variance_cl2}}}{\le}
\pr{a_{\Delta} (n-k)p}^{-2\ell}e^{O(1/p)}k^3\pr{\frac{a_{\Delta}np}{\ell}}^{\ell}
\pr{(2\ln(enp))^{\Delta}}^{2\min\{\ell,k-\ell\}}
\times\\
&\quad\quad\quad\quad\quad\quad\quad\quad\quad\quad\quad
\times q^{-\ell^2/2-2\ell(k-\ell)+\ell/2+2k\ell-\ell^2-3\ell} \sum\limits_{m=1}^{\ell}\frac{\ell^{m+\ell}}{(m!)^2}
 \pr{\frac{p a^2_{\Delta}}{q e^2}}^{-m}.
\end{align*}
Observe that
$$
 \sum\limits_{m=1}^{\ell}\frac{\ell^{m+\ell}}{(m!)^2}
 \pr{\frac{p a^2_{\Delta}}{q e^2}}^{-m}\leq \ell^\ell\left(\sum\limits_{m=1}^{\infty}\frac{\ell^{m/2}}{m!}
 \pr{\sqrt{\frac{q e^2}{p a^2_{\Delta}}}}^m\right)^2=\ell^{\ell}\exp\left(2\sqrt{\frac{\ell q e^2}{pa^2_{\Delta}}}\right).
$$
 Therefore,
$$
\frac{F_{\ell}}{(\mathbb{E} Z_k)^2} \leq\left(\frac{n}{n-k}\right)^{\ell}\pr{a_{\Delta}(n-k)p}^{-\ell}e^{O(1/p)}k^{3}e^{2\sqrt{\frac{\ell q e^2}{pa^2_{\Delta}}}}
\pr{(2\ln(enp))^{2\Delta}}^{\min\{\ell,k-\ell\}} q^{\ell^2/2-5\ell/2}
$$
Since $\varepsilon>0$ is small, we get that
$$
q^{(k-3)/2}\leq a_{\Delta}np(1-\varepsilon)\leq (1-\varepsilon/2)\left(\frac{n-k}{n}\right)^2a_{\Delta}np=(1-\varepsilon/2)\left(\frac{n-k}{n}\right)a_{\Delta}(n-k)p.
$$
So,
\begin{align*}
\frac{F_{\ell}}{(\mathbb{E} Z_k)^2} &\leq k^{3}\pr{q^{\frac{\ell}{2}-5/2}q^{-\frac{k-3}{2}}(1-\varepsilon/2)\pr{(2\ln(enp))^{2\Delta}}^{\min\left\{1,\frac{k-\ell}{\ell}\right\}}\cdot e^{\frac{O(1)}{\sqrt{\ell p}}+\frac{O(1)}{\ell p}}}^{\ell}
\\
&\le k^3(1-\varepsilon)^{\ell/2}\pr{q^{-(k-\ell)/2}\pr{(2\ln(enp))^{2\Delta}}^{\big(1\wedge \frac{k-\ell}{\ell}\big)}}^{\ell},
\end{align*}
where the last inequality holds since, due to the restrictions on $\ell$, $\ell p>\sqrt{kp}$, implying
$$
\sqrt{1-\varepsilon}\cdot e^{\frac{O(1)}{\sqrt{\ell p}}+\frac{O(1)}{\ell p}}<\sqrt{1-\varepsilon/2}<1
$$
when $C_{\varepsilon}$ is large enough, according to~\eqref{eq:kp_large}. Finally, we get
\begin{align}
\frac{F_{\ell}}{(\mathbb{E} Z_k)^2}
&\le k^3(1-\varepsilon)^{\ell/2}\pr{q^{-1/2}\pr{2(\ln(enp))^{2\Delta}}^{\min\left\{\frac{1}{k-\ell},\frac{1}{\ell}\right\}}}^{\ell(k-\ell)}\notag\\
&\le k^3(1-\varepsilon)^{\ell/2}\pr{q^{-1/2}(2\ln(enp))^{4\Delta/k}}^{\ell(k-\ell)}\le k^3(1-\varepsilon)^{\sqrt{k}/2},
\label{eq:variance_forests_small_3}
\end{align}
where the last inequality follows from $q^{-1/2}(2\ln(enp))^{4\Delta/k}=\sqrt{1-p}\cdot\exp\left(\frac{O(\ln\ln np)}{\Omega(\ln n)}\right)=o(1)$.\\

From~\eqref{eq:forests_sparse_exp_large},~\eqref{eq:variance_forests_small_1},~and~\eqref{eq:variance_forests_small_3}, we conclude that, if $n^{-1/2+\varepsilon}<p<1-\varepsilon$, then
$$
\frac{\mathrm{Var}Z_k}{\pr{\mathbb{E} Z_k}^2}
\le\sum\limits_{\ell=1}^{k}\frac{F_\ell}{\pr{\mathbb{E} Z_k}^2}+\frac{1}{\mathbb{E} Z_k}
\le \sum\limits_{\ell=1}^{\infty}n^{-\varepsilon\ell}+k^4(1-\varepsilon)^{\sqrt{k}/2}+o(1)=o(1).
$$
For $C_{\varepsilon}/n\leq p\leq n^{-1/2+\varepsilon}$, from~\eqref{eq:forests_sparse_exp_large},~\eqref{eq:variance_forests_small_2},~and~\eqref{eq:variance_forests_small_3}, we get
$$
\frac{\mathrm{Var} Z_k}{\pr{\mathbb{E}Z_k}^2}
\le \sum\limits_{\ell=1}^{\left\lfloor \sqrt{k/p}\right\rfloor}\exp\pr{\frac{k}{5(\ln(np))^5}}+k^4(1-\varepsilon)^{\sqrt{k}/2}+o(1)
\le\exp\pr{\frac{O(k)}{(\ln(np))^5}},
$$
completing the proof of Lemma~\ref{lm:variance_sparse_forests}.
\end{proof}

\subsection{Proof of Theorem~\ref{th:concentration_max_ind_rooted_forests_bounded_degree}}
\label{sc:small_p_proof_Talagrand}

From Claim~\ref{claim:expectation_n_ind_forests_bounded_degree_at_k_pm_eps} and Markov's inequality, we immediately get that whp $\mathbf{F}_{\Delta}\leq k_{+\varepsilon}-1$.\\

If $p\ge n^{-1/2+\varepsilon}$, then Lemma~\ref{lm:variance_sparse_forests} and Chebyshev's inequality imply that whp $\mathbf{F}_{\Delta}\geq k_{-\varepsilon}$. It remains to prove that the same holds when $p<n^{-1/2+\varepsilon}$. We will use Talagrand's inequality (see, e.g.,~\cite[Theorem 2.29]{Janson2000}). Let us recall its statement.

Let $\Omega=\prod_{i=1}^n\Omega_i$ be endowed with the product measure. A function $X:\Omega\to\mathbb{R}$ is Lipschitz if $|X(x)-X(y)|\leq 1$ for all $x,y\in\Omega$ that are at Hamming distance at most 1. For $f:\mathbb{R}\to\mathbb{R}$, let us say that $X$ is {\it $f$-certifiable} if, for any $x\in\Omega$ and $r\in\mathbb{R}$ such that $X(x)\geq r$, there exists $I\subset[n]$ such that $|I|\leq f(r)$ and each $y\in\Omega$ coinciding with $x$ in coordinates from $I$ satisfies $X(y)\geq r$ as well.

\begin{theorem}[Talagrand]
If a random variable $X$ is Lipschitz and $f$-certifiable, then, for every $r\in\mathbb{R}$ and $t\geq 0$, $\mathbb{P}(X\leq r-t)\mathbb{P}(X\geq r)\leq e^{-\frac{t^2}{4f(r)}}.$
\label{th:Talagrand}
\end{theorem}

If we let $\Omega_i=\{\{j,i\},j\in[i-1]\}$, for $i\in\{2,\ldots,n\}$, then $\mathbf{F}_{\Delta}$ is Lipschitz and $\ceil{r}$-certifiable. Indeed, for any integer $r$, there exists a certificate of size $r$ enforcing the inequality $\mathbf{F}_{\Delta}\ge r$, which is simply the set of vertices of the corresponding largest induced forest. Thus, by Theorem~\ref{th:Talagrand}, we get, for large enough $n$,
\begin{align*}
\Pb{\mathbf{F}_{\Delta}\le{k_{-\varepsilon}-1}}\Pb{\mathbf{F}_{\Delta}\ge{k_{-\varepsilon/2}}}
&\le\exp\pr{-\frac{(k_{-\varepsilon/2}-k_{-\varepsilon}+1)^2}{4k_{-\varepsilon/2}}}\leq
\exp\left(-\frac{(2\log_q\frac{1-\varepsilon/2}{1-\varepsilon})^2}{4k_{-\varepsilon/2}}\right)\\
&\le\exp\pr{-\frac{\varepsilon^2}{4k_{-\varepsilon/2}p^2}}\le\exp\pr{-\frac{\varepsilon^{2}k_{-\varepsilon/2}}{40(\ln(np))^2}}.
\end{align*}
Using Lemma~\ref{lm:variance_sparse_forests} with $\varepsilon$ replaced by $\varepsilon/2$ and the Paley--Zygmund inequality, we get
$$
\Pb{\mathbf{F}_{\Delta}\ge{k_{-\varepsilon/2}}}=\Pb{Z_{k_{-\varepsilon/2}}>0}\ge\frac{\left(\mathbb{E} Z_{k_{-\varepsilon/2}}\right)^2}{\mathbb{E}\left(Z_{k_{-\varepsilon/2}}^2\right)}=
\frac{\left(\mathbb{E} Z_{k_{-\varepsilon/2}}\right)^2}{\mathrm{Var} Z_{k_{-\varepsilon/2}}+\left(\mathbb{E} Z_{k_{-\varepsilon/2}}\right)^2}
\ge\exp\pr{-\frac{O(k_{-\varepsilon/2})}{(\ln(np))^5}}
$$
and, finally, recalling that $p<n^{-1/2+\varepsilon}$,
$$
\Pb{\mathbf{F}_{\Delta}\le{k_{-\varepsilon}-1}}
\le\exp\pr{\frac{O(k_{-\varepsilon/2})}{(\ln(np))^5}-\frac{\varepsilon^{2}k_{-\varepsilon/2}}{40(\ln(np))^2}}
<\exp\pr{-\frac{n^{1/2-\varepsilon}}{(\ln(np))^2}}=o(1),
$$
completing the proof.


\section{Acknowledgements}

The research work of V. Kozhevnikov is supported by Grant  NSh-775.2022.1.1 for Leading Scientific Schools.  

\end{document}